\documentclass{article}
\usepackage{amsmath, amssymb, amsthm, amsfonts}
\usepackage{empheq}
\newcommand{\si}{\sigma}
\newcommand{\la}{\lambda}
\newcommand{\ol}{\overline}
\newcommand{\pa}{\partial}

\newcommand{\al}{\alpha}
\newcommand{\be}{\beta}
\newcommand{\de}{\delta}

\newcommand{\ka}{\kappa}

\newcommand{\om}{\omega}
\newcommand{\ve}{\varepsilon}
\newcommand{\ze}{\zeta}
\newcommand{\Om}{\Omega}
\newcommand{\cd}{\cdot}
\newcommand{\vrho}{\varrho}

\newcommand{\diag}{\operatornamewithlimits{diag}}
\newcommand{\R}{{\mathbb R}}
\newcommand{\C}{{\mathbb C}}

\newcommand{\N}{{\mathbb N}}

\newcommand{\cA}{{\cal A}}
\newcommand{\cE}{{\cal E}}

\newcommand{\cH}{{\cal H}}

\newcommand{\cM}{{\cal{M}}}

\renewcommand{\(}{\left(}
\renewcommand{\)}{\right)}
\newcommand{\Th}{\varTheta}

\newcommand{\tp}{\tilde{p}}

\newcommand{\tde}{\tilde{\delta}}

\newcommand{\tla}{\tilde{\lambda}}

\newcommand{\tr}{\tilde{r}}
\newcommand{\tA}{\tilde{A}}
\newcommand{\tR}{\tilde{R}}
\newcommand{\tV}{\tilde{V}}
\newcommand{\tM}{\tilde{M}}
\newcommand{\tLa}{\tilde{\varLambda}}
\newcommand{\tPhi}{\tilde{\varPhi}}

\renewcommand{\th}{\theta}

\newtheorem{theorem}{\bf Theorem}[section]

\newtheorem{lemma}[theorem]{\bf Lemma}
\newtheorem{proposition}[theorem]{\bf Proposition}

\theoremstyle{remark}
  
\theoremstyle{definition}
  \newtheorem{definition}[theorem]{Definition}
  \newtheorem{example}[theorem]{\sc Example}
\numberwithin{equation}{section}

\begin{document}

\title{
On the energy decay estimate for the dissipative wave equation with very fast oscillating coefficient and smooth initial data}
\author{
Kazunori GOTO \!\!
\footnote{Takasago High School of Hyogo, Japan; e-mail: b003vbv@y-u.jp}
\; and \;
Fumihiko HIROSAWA \!\!
\footnote{
Department of Mathematical Sciences, Faculty of Science, Yamaguchi University, Japan; 
e-mail: hirosawa@yamaguchi-u.ac.jp}
}

\date{}
\maketitle

\begin{abstract}
In this paper we consider energy decay estimates for the Cauchy problems of dissipative wave equations with time dependent coefficients, 
in particular, the coefficients consisting of weak dissipation and very fast oscillating terms. 
For such a problem, which have been difficult to deal with in previous research, we prove energy decay estimates by introducing a new condition for the coefficients to evaluate oscillating cancellation of the energy, and smooth initial data such as in the Gevrey class.
\end{abstract}

%%%%%%%%%%%%%%%%%%%%%%%%%%%%%%%%%%%
%
\section{Introduction}
%
%%%%%%%%%%%%%%%%%%%%%%%%%%%%%%%%%%%
Let us consider the following Cauchy problem for a dissipative wave equation with a time dependent coefficient: 
\begin{equation}\label{u}
\begin{dcases}
  \(\pa_t^2 - \Delta + b(t)\pa_t\)u(t,x) = 0, 
  & (t,x) \in (0,\infty) \times \R^n,\\
  u(0,x) = u_0(x),\;\;
  \pa_t u(0,x) = u_1(x), 
  & x \in \R^n,
\end{dcases}
\end{equation}
where 
$\Delta = \sum_{j=1}^n \pa_{x_j}^2$. 
Then the total energy of the solution of \eqref{u} at $t \in \R_+:=[0,\infty)$ is defined by 
\begin{equation*}
  E(t;u):=\frac12\(
  \left\|\nabla u(t,\cd)\right\|^2 + \left\|\pa_t u(t,x)\right\|^2
  \), 
\end{equation*}
where 
$\nabla = {}^t(\pa_{x_{1}},\ldots,\pa_{x_n})$ 
and 
$\|\;\;\|$ denotes the usual $L^2(\R^n)$ norm. 

Multiplying the equation of \eqref{u} by $\pa_t u(t,x)$ and integrating over $\R^n$ gives us the following equality: 
\begin{equation}\label{E-u}
  \frac{d}{dt} E(t;u) = -b(t)\left\|\pa_t u(t,\cd)\right\|^2. 
\end{equation}
The dissipative term $b(t)\pa_t u(t,x)$ in equation \eqref{u} represents the resistance force opposite to the direction of motion, so the coefficient $b(t)$ is usually assumed to be $b(t)\ge 0$. 
In that case, we see from \eqref{E-u} that $E(t;u)$ decreases monotonically with respect to $t$. 
Moreover, the following facts are trivial by $\|\pa_t u(t,\cd)\| \le 2E(t;u)$ and Gronwall's inequality. 
%%%%%%%%%%%%%%%%%%%%%%%%%%
\begin{itemize}
\item 
If $b(t) \equiv 0$, that is, equation in \eqref{u} is free wave equation, then the energy conservation $E(t;u)\equiv E(0;u)$ holds. 
\item 
$E(t;u) \ge \exp(-2\int^t_0 b(s)\,ds) E(0;u)$ is valid, it follows that the decay order of $E(t;u)$ is at most $\exp(-2\int^t_0 b(s)\,ds)$. 
\item 
If $\int^\infty_0 b(s)\,ds < \infty$ and $E(0;u)>0$, then 
$\lim_{t\to\infty} E(t;u)>0$, that is, $E(t;u)$ does not decay to $0$ as $t\to\infty$. 
\end{itemize}
%%%%%%%%%%%%%%
Taking them into consideration, the problem we should consider are whether 
$\lim_{t\to\infty}E(t;u)=0$ holds under the assumption 
$\int_0^\infty b(s)\,ds = \infty$, 
and further whether the decay order is $\exp(-2\int^t_0 b(s)\,ds)$ as $t\to\infty$. 
For a positive and monotone decreasing function $\eta(t)$ satisfying 
\[
  \lim_{t\to\infty}\eta(t)=0 
\]
and a positive constant $E_0=E_0(u_0,u_1)$, 
we call the estimate 
\begin{equation}\label{decay0}
  E(t;u) \lesssim \eta(t) E_0 
\end{equation}
the energy decay estimate, 
where $f \lesssim g$ for positive functions $f$ and $g$ denotes that there exists a positive constant $C$ such that $f \le C g$. 
Moreover, $f \simeq g$ denotes that if $f \lesssim g$ and $g\lesssim f$ hold. 
Our main interest in the problem of energy decay estimate for \eqref{u} is the relation between the dissipative coefficient $b(t)$ and the decay order $\eta(t)$. 
There are many results on this type of problem, and we will introduce some representative results that are closely related to the research topic of this paper. In the results presented below, $E_0$ in \eqref{decay0} is given by 
\begin{equation}\label{E0}
  E_0 = E(0;u) + \left\|u_0(\cd)\right\|^2.
\end{equation}

If $b(t)$ is a positive constant, then \eqref{E0} with $\eta(t)=(1+t)^{-1}$ is valid (\cite{M76}). 
If $b(t)$ is given by 
\begin{equation}\label{mu0}
  b(t) = \mu_0(1+t)^{-1}
\end{equation}
with a positive constant $\mu_0$, which is called weak dissipation, 
then \eqref{E0} with 
\begin{equation}\label{eta-mu0}
  \eta(t)=(1+t)^{-\min\{2,\mu_0\}}
\end{equation}
is valid. 
For more general $b(t)\in C^1(\R_+)$ and $\mu_0\in(0,2]$, 
if the following conditions hold: 
\begin{equation}\label{AU}
  b(t) \ge \mu_0(1+t)^{-1}
  \;\text{ and }\;
  (\mu_0-1)(\mu_0-2)-(\mu_0-1)(1+t)b(t)-(1+t)^2b'(t) \ge 0,
\end{equation} 
where \eqref{AU} corresponds to $\mu_0 \le 2$ in the case \eqref{mu0}, 
then \eqref{decay0} with $\eta(t)=(1+t)^{-\mu_0}$ is also valid (\cite{M77, U80}). 
These classical results are proved by so-called energy method and can be generalized to cases where the coefficient depends on not only time variable but also spatial variables. 

If the dissipative coefficient depends only on $t$, as in our model, more precise considerations have been done in \cite{GG25, HW08, W04, W06, W07}. 
According to these, the characteristics of the equation change significantly depending on the order of $b(t)$ as $t\to\infty$. 
Specifically, when the rate of decrease is slower than $(1+t)^{-1}$, the effect of the dissipative term becomes larger and the equation becomes parabolic-like, so that the decay order is given by $1/(\int^t_0 1/b(s)ds)$. 
A dissipative term with such a strong effect is called effective dissipation, whereas if $b(t)$ behaves as in \eqref{mu0}, the dissipation is called non-effective. 
As we consider a non-effective dissipation in this paper, this problem will be discussed in more detail below.

The typical dissipative coefficient $b(t)$ we consider is the following function of weak dissipation \eqref{mu0} with an oscillating error term $\si(t)$: 
\begin{equation}\label{mu0-si}
  b(t) = \mu_0(1+t)^{-1} + \si(t).
\end{equation}
The effect of $\si(t)$ has been studied in \cite{W04, W06}, 
and roughly speaking, the effect can be ignored for the decay estimate \eqref{decay0} with \eqref{eta-mu0} if the following conditions hold: 
\begin{equation}\label{g-zero}
  \sup_{t>0}\left\{\int^t_0\si(s)\,ds\right\}<\infty,
\end{equation}
\begin{equation}\label{VSO}
  \si(t)=o\(t^{-1}\)
  \;\text{ and }\;
  |\si'(t)|=O\(t^{-2}\) 
  \;\;(t\to\infty). 
\end{equation}
We call very slow oscillation (=VSO) if \eqref{VSO} holds. 
In general, VSO for the dissipative coefficient is considered to be essential for the energy estimates. 
In fact, in \cite{RS05}, an example is presented that the energy is unbounded with respect to $t$ in the following wave equation with time dependent propagation speed, which is obtained by transforming \eqref{u}, when the corresponding VSO do not hold: 
\begin{equation}\label{uw}
  \(\pa_t^2 - a(t)^2 \Delta\)w(t,x) = 0, 
\end{equation}
where $a(t)\in C^2(\R_+)$, 
$\inf_t\{a(t)\}>0$ and $\sup_t\{a(t)\}<\infty$. 
However, VSO in \eqref{uw} is not necessarily necessary if we add the following condition, which is called stabilization condition introduced in \cite{H07}: 
\begin{equation}\label{stb-a}
  \text{there exists $a_\infty \ge 0$ such that }\;
  \int^t_0 |a(s)-a_\infty|\,ds = o(t) \;\; (t\to\infty). 
\end{equation}
The following theorem was obtained by introducing conditions corresponding to \eqref{stb-a} into the problem of energy decay estimate for \eqref{u}. 
%
%%%%%%%%%%%%%%%%%%%%%%%%%%%%
\begin{theorem}[Theorem 5.1 \cite{HW08}]\label{Thm-HW}
Let $b(t)\in C^m(\R_+)$ with $m \ge 1$. 
If there exist real numbers $\mu_0$, $\al$ and $\om_\infty$ satisfying 
$0<\mu_0<1$ and $0 \le \al <1$ such that 
$\si(t)$ defined by \eqref{mu0-si} satisfies 
\eqref{g-zero}, 
\begin{equation}\label{HW-A1}
  \int^t_0\left|
    \exp\(\int^s_0 \si(\tau)\,d\tau\) - \om_\infty
  \right|\,ds
  \lesssim (1+t)^{\al}
\end{equation}
and 
\begin{equation}\label{HW-A2}
  \left|\si^{(k)}(t)\right| \lesssim (1+t)^{-(k+1)\be}
  \;\; (k=0,\ldots,m)
\end{equation}
for 
\begin{equation}\label{HW-albe}
  \be \geq \be_0=\be_0(\al,m) := \frac{m}{m+1} \al + \frac{1}{m+1},
\end{equation}
then \eqref{decay0} with $\eta(t)=(1+t)^{-\mu_0}$ is valid. 
\end{theorem}
%%%%%%%%%%%%%%%%%%%%%%%%%%%%

Here we note that \eqref{HW-A1} with $\al=1$ is trivial by \eqref{g-zero}. 
If $\al=1$ or $m=0$, then $\be_0=1$, which is approximately VSO. 
On the other hand, if $\al<1$ and $m\ge1$, then $\be_0<1$, which is definitely not VSO. 
Moreover, $\be_0$ is monotone decreasing with respect to $m$. 
Thus, by adding the assumption of higher differentiability and decay of higher derivatives for $\si(t)$, $\be_0$ becomes smaller; that is, 
if we focus only on the condition $k=0$ in \eqref{HW-A2}, 
the restriction on the order of the amplitude of $\si(t)$ is weakened. 
An interesting aspect of this result is that the order of the oscillating term $\si(t)$ need not necessarily be less than the order of the principal part $\mu_0(1+t)^{-1}$. 
It follows that $b(t)$ can be negative, and then \eqref{E-u} means that the energy increases, 
but the order of $\eta(t)$ is consequently the same as in the case $\si(t)\equiv0$. 

The assertion of the main theorems of this paper, which will be introduced in the next chapter, is that the energy decay estimate \eqref{decay0} can be obtained though the order of the amplitude of $\si(t)$ does not satisfy the assumption of Theorem \ref{Thm-HW}. 
The new approaches to achieving this is to extend the condition \eqref{HW-A1} and introduce a new function space for the initial data. 

%%%%%%%%%%%%%%%%%%%%%%%%%%%%%%%%%%%
%
\section{Main theorems and examples}
%
%%%%%%%%%%%%%%%%%%%%%%%%%%%%%%%%%%%

%%%%%%%%%%%%%%%%%%%%%%%%%%%%%%%%%%%
\subsection{Main theorems}
%%%%%%%%%%%%%%%%%%%%%%%%%%%%%%%%%%%

Let us introduce the following three properties for the dissipative coefficient $b(t)$ based on the assumptions introduced in \cite{HW08}. 
\begin{description}
\item[(A1)] Non-effective damping condition: 
There exist $T \ge 0$ and $\mu(t)$ on $\R_+$ satisfying 
\begin{equation*}%\label{mu}
  \sup_{t \ge T} \{(1+t)\mu(t)\}=:\mu_0<1, \;\;
  \mu(t)>0
  \;\text{ and }\;
  \mu'(t)<0
  \;\text{ for any $t \ge T$}
\end{equation*}
such that the $\int^\infty_0(b(t)-\mu(t))\,dt$ converges; 
we shall call $\mu(t)$ the principal part of $b(t)$. 
\item[(A2)] Stabilization condition: 
The oscillating part of $b(t)$ defined by 
\begin{equation*}%\label{def_si}
  \si(t):=b(t)-\mu(t)
\end{equation*}
satisfies 
\begin{equation*}%\label{stb0}
  \int^\infty_0 \left|\int^\infty_s \si(\tau)\,d\tau\right|\,ds<\infty. 
\end{equation*}
\item[(A3)] $C^m$-property: $b(t) \in C^m(\R_+)$. 
\end{description}
Here we note that
(A1) is a generalization of \eqref{mu0-si} with $\mu_0<1$ and 
(A2) is the assumption corresponding to the case $\al=0$ in \eqref{HW-A1} with $\om_\infty = \exp(\int_0^\infty \si(\tau)\,d\tau)$. 

%%%%%%%%%%%%%%%%%%%%%%%%%%%%%%%%%%
%%%%%%%%%%%%%%%%%%%%%%%%%%%%%%%%%%
\begin{theorem}\label{Thm0} 
Let $(u_0,u_1) \in H^1 \times L^2$, and $b(t)$ satisfy 
$\mathrm{(A1)}$, $\mathrm{(A2)}$ and $\mathrm{(A3)}$. 
If there exist a positive and strictly increasing continuous function $\Th(t)$ and a positive and monotone continuous function $\varXi(t)$ such that 
\begin{equation}\label{stb}
  \int^\infty_t \left|\int^\infty_s \si(\tau)\,d\tau\right|\,ds
  \le \Th(t)^{-1},
\end{equation}
\begin{equation}\label{Cm}
  \left|\frac{d^k}{dt^k}b(t)\right| \le \varXi(t)^{k+1}
  \;\; (k=0,\ldots,m),
\end{equation}
\begin{equation}\label{XiTh-thm0}
  \sup_{t\ge0} \left\{\frac{\varXi(t)}{\Th(t)}\right\} < \infty
\end{equation}
and
\begin{equation}\label{ass0}
  \sup_{t\ge0} \left\{
  \Th(t)^{-m} \int^t_0 \varXi(s)^{m+1}\,ds\right\}
  <\infty, 
\end{equation}
then the following energy decay estimate with \eqref{E0} is established: 
\begin{equation}\label{decay-thm}
  E(t;u) \lesssim E_0 \exp\(-\int^t_0 \mu(s)\,ds\). 
\end{equation} 
\end{theorem}
%%%%%%%%%%%%%%%%%%%%%%%%%%%%%%%%%%
%%%%%%%%%%%%%%%%%%%%%%%%%%%%%%%%%%

The comparison of Theorem \ref{Thm0} and Theorem \ref{Thm-HW} will be explained later with specific examples, but generally speaking, 
\eqref{stb} is an extension of \eqref{HW-A1} to the case $\al<0$, 
\eqref{Cm} corresponds to \eqref{HW-A2}, 
\eqref{XiTh-thm0} is the assumption corresponding to $\al \le \be_0$, which is clear from \eqref{HW-albe}, and 
\eqref{ass0} is the assumption corresponding to \eqref{HW-albe}. 

In previous studies on the energy decay for linear dissipative wave equations, the initial data was generally set to $(u_0,u_1) \in H^1 \times L^2$. 
In contrast, the second main theorem of this paper asserts that if the initial data belong to a more restricted function space than $H^1 \times L^2$, it may be possible to obtain some kind of energy decay estimate though the order of the amplitude of $\si(t)$ does not satisfy the assumptions of Theorem \ref{Thm0}.

Let $\rho$ be a positive and strictly increasing continuous function on $\R_+$ and satisfies that 
$\lim_{r\to\infty} \rho(r)=\infty$ 
and 
\begin{equation}\label{zeta}
  \ze(r):= \frac{r}{\rho(r)}
\end{equation}
is strictly increasing; 
such a function $\rho$ is referred to as a weight function here. 
Let us define the Fourier transform of $f$ with respect to the spatial variables as follows and denote it as $\hat{f}$: 
\[
  \hat{f}(\xi):=(2\pi)^{-\frac{n}{2}}\int_{\R^n} e^{-ix\cd\xi}f(x)\,dx. 
\]
For a positive real number $\ka$ and a weight function $\rho$, we define the space of initial data $\cH(\rho,\ka)$ by 
\begin{equation*}%\label{cH}
  \cH(\rho,\ka):=\left\{(u_0,u_1)\in H^1 \times L^2\;;\;
    \left\|(u_0,u_1)\right\|_{\cH(\rho,\ka)}<\infty
  \right\},
\end{equation*}
where 
\begin{equation*}
  \left\|(u_0,u_1)\right\|_{\cH(\rho,\ka)}
  :=
  \(\int_{\R^n} e^{2\ka\rho(|\xi|)}
  \(\(1+|\xi|^2\)|\hat{u}_0(\xi)|^2 + |\hat{u}_1(\xi)|^2\)\,d\xi\)^{\frac12}.  
\end{equation*}
Moreover, we define $\cH(\rho,\infty)$ by 
\[
  \cH(\rho,\infty):=\bigcap_{\ka>0}\cH(\rho,\ka). 
\]
Here the following lemma is trivial from the definition of $\cH(\rho,\ka)$: 
%%%%%%%%%%%%%%%%%%%%%%%%%%%
\begin{lemma}\label{Lemm-cH}
The following properties are valid: 
\begin{itemize}
\item[(i)] 
If $0<\ka_1<\ka_2 \le \infty$, then 
$\cH(\rho,\ka_2) \subset \cH(\rho,\ka_1)$. 
\item[(ii)] 
If the weight functions $\rho$ and $\tilde{\rho}$ satisfy 
$\lim_{r\to\infty}\rho(r)/\tilde{\rho}(r)=0$, then 
$\cH(\tilde{\rho},\ka) \subset \cH(\rho,\infty)$ for any $\ka>0$. 
\item[(iii)] 
If $\rho(r)=\log(e+r)$, then $\cH(\rho,\ka)=H^{\ka+1} \times H^\ka$, where $H^\ka$ is the usual Sobolev space of order $\ka$. 
\item[(iv)] 
Let $\nu>1$ and $\rho_\nu(r) := (1+r)^{1/\nu}$. 
If $u_0$ and $u_1$ are support compacted Gevrey class of order $\nu$, then there exists $\ka>0$ such that 
$(u_0,u_1) \in \cH(\rho_\nu,\ka)$ (see \cite{Ro93}). 
\end{itemize}
\end{lemma}
%%%%%%%%%%%%%%%%%%%%%%%%%%%

%%%%%%%%%%%%%%%%%%%%%%%%%%%%%%%%%%%%%%%
\begin{theorem}\label{Thm1} 
Let $(u_0,u_1) \in \cH(\rho,\infty)$ for a weight function $\rho$, 
and $b(t)$ satisfy $\mathrm{(A1)}$, $\mathrm{(A2)}$ and $\mathrm{(A3)}$. 
If there exist a positive and strictly increasing continuous function $\Th(t)$ and a positive and monotone continuous function $\varXi(t)$ such that 
\eqref{stb}, \eqref{Cm}, 
\begin{equation}\label{XiTh-thm1}
  \sup_{t\ge0} \left\{\frac{\varXi(t)}{\ze^{-1}\(\Th(t)\)}\right\} < \infty
\end{equation}
and
\begin{equation}\label{ass1}
  \sup_{t\ge0} \left\{
  \frac{\Th(t)}{\ze^{-1}\(\Th(t)\)^{m+1}}
  \int^t_0 \varXi(s)^{m+1}\,ds\right\}
  <\infty, 
\end{equation}
where $\ze$ is defined by \eqref{zeta}, 
then there exists $\ka_0>0$ such that the energy decay estimate \eqref{decay-thm} with $E_0=\|(u_0,u_1)\|^2_{\cH(\rho,\ka_0)}$ is established. 
\end{theorem}
%%%%%%%%%%%%%%%%%%%%%%%%%%%%%%%%%%%%%%%

%%%%%%%%%%%%%%%%%%%%%%%%%%%%%%%%%%%
\subsection{Examples}
%%%%%%%%%%%%%%%%%%%%%%%%%%%%%%%%%%%
Let us consider the conditions under which Theorem \ref{Thm0} and Theorem \ref{Thm1} can be applied to specific examples of $\Th(t)$, $\varXi(t)$ and $\ze(r)$.  

%%%%%%%%%%%%%%%%%%%%%%%%%%%%%%%%%%%%%%%%%%%%
\begin{example}\label{Ex0}
Let $m$ be a positive integer. 
For $\al<0$, $\be \le 1$, $\be \neq 1/(m+1)$ and $\nu>1$, we define 
$\Th(t)$, $\varXi(t)$ and $\ze(r)$ by 
\begin{equation}\label{ThXize}
  \Th(t)=(1+t)^{-\al},\;\;
  \varXi(t)=(1+t)^{-\be}
  \;\text{ and }\;
  \ze(r)=r (1+r)^{-\frac{1}{\nu}}, 
\end{equation}
that is, $\rho(r) = \rho_\nu(r)$, where $\rho_\nu$ is defined in Lemma \ref{Lemm-cH}. 
\begin{itemize}
\item[(i)] 
If $\al m - \be (m+1)+ 1 \le 0$, that is, 
\begin{equation}\label{Ex_al-be-1}
   \be \ge \be_0 = \be_0(\al,m)
   :=\frac{m}{m+1}\al + \frac{1}{m+1}, 
\end{equation}
then \eqref{XiTh-thm0} and \eqref{ass0} are valid, hence Theorem \ref{Thm0} is applicable. 
\item[(ii)] 
If $\al m - \be (m+1) + 1 + \al(m+1)/(\nu-1)\le 0$, that is, 
\begin{equation*}%\label{Ex_al-be-2}
  \be  \ge \tilde{\be_0} = \tilde{\be_0}(\al,m,\nu)
  :=\(\frac{m}{m+1}+\frac{1}{\nu-1}\)\al + \frac{1}{m+1},
\end{equation*}
then \eqref{XiTh-thm1} and \eqref{ass1} are valid, hence Theorem \ref{Thm1} is applicable. 
\end{itemize}
Since \eqref{Ex_al-be-1} is the same as \eqref{HW-albe}, Theorem \ref{Thm0} and Theorem \ref{Thm1} can be considered as extensions of Theorem \ref{Thm-HW} for $\al<0$. 
Moreover, for $\al<0$, we note the followings. 
\begin{itemize}
\item 
$\be_0 > \tilde{\be}_0$. 
\item 
$\be_0$ and $\tilde{\be}_0$ are monotone decreasing with respect to $m$ and monotone increasing with respect to $\al$.  
\item 
$\tilde{\be}_0$ is monotone increasing with respect to $\nu$ and $\lim_{\nu\to\infty}\tilde{\be}_0=\be_0$. 
\end{itemize}
\end{example}

Next, we will introduce some concrete examples of the dissipative coefficient $b(t)$ that can be applied to Theorem \ref{Thm0} and Theorem \ref{Thm1}. 
For simplicity, we will consider specific examples of $\si(t)$ in the case of \eqref{mu0-si} with $0<\mu_0<1$, and paying particular attention to the order of the amplitude of $\si(t)$ described by the parameter $p$ in the following examples. 

%%%%%%%%%%%%%%%%%%%%%%%%%%%%%%%%%%%%%%%%%%%%
\begin{example}\label{Ex1}
For $p\ge -1$, $q > 1$ and a positive integer $m$, 
we define $\si_1(t)$, $\Th_1(t)$ and $\varXi_1(t)$ on $\R_+$ by 
\begin{equation}\label{si-sin}
  \si_1(t):= (1+t)^p \sin (1+t)^q, 
\end{equation}
\begin{equation*}
  \Th_1(t): = (1+t)^{-\al_1}, \;\;
  \al_1 = p - q + 2,
\end{equation*}
\begin{equation*}
  \varXi_1(t): = (1+t)^{-\be_1}
  \;\text{ and }\;
  \be_1 = -q + 1 + \frac{-p+q-1}{m+1}, 
\end{equation*}
where $\be_1 \neq 1/(m+1)$. 
Let $b(t)$ be given by \eqref{mu0-si} with $\si(t)=\si_1(t)$, and suppose that 
\begin{equation*}
  p < q - 2. 
\end{equation*}
Then, for any positive integer $m$, there exists a positive constant $C_1$ such that \eqref{stb} and \eqref{Cm} are valid 
for $\Th(t)=C_1\Th_1(t)$ and $\varXi(t)=C_1\varXi_1(t)$ 
by Lemma \ref{lemm1-Ex1} and Lemma \ref{lemm2-Ex1} in Appendix. 
Moreover, if 
\begin{equation*}
  \al_1 m - \be_1 (m+1)+ 1 = (p+1)(m+1) \le 0, 
\end{equation*}
that is, 
\begin{equation}\label{Ex1_pq-1}
  p \le p_1:=-1, 
\end{equation}
then \eqref{XiTh-thm0} and \eqref{ass0} are valid by Example \ref{Ex0} (i), and thus Theorem \ref{Thm0} can be applicable. 
\end{example}
%%%%%%%%%%%%%%%%%%%%%%%%%%%%%%%%%%%%%%%%%%%%

When applying $\si(t)$ defined in Example \ref{Ex1} to Theorem \ref{Thm0}, Example \ref{Ex0} (i) requires \eqref{Ex1_pq-1}, it follows that $|\si(t)|\lesssim (1+t)^{-1}$. 
Thus $\si(t)$ must be almost VSO and the order is the same or less than the principal part $\mu_0(1+t)^{-1}$. 
In such a case, many results are already known and Theorem \ref{Thm0} does not provide any new advances. 
For example, \cite{AE23} and \cite{GG25} provide more detailed consideration of $\si(t)$, which is closer to VSO. 
In \cite{AE23}, the relationship between a more general principal part, which can be introduce logarithmic order, an oscillating part and the decay order of the energy is studied. 
In \cite{GG25}, the energy decay estimate for the case that 
$(1+t)|\si(t)| \le \mu_0$ is investigated in detail. 

%%%%%%%%%%%%%%%%%%%%%%%%%%%%%%%%%%%%%%%%%%%%
\begin{example}\label{Ex11}
Let $\si(t)$, $\Th(t)$ and $\varXi(t)$ be the same as in Example \ref{Ex1}, 
and $\rho(r)=\rho_\nu(r)$ with $\nu>1$. 
Since $\ze(r) \simeq (1+r)^{(\nu-1)/\nu}$, if 
\begin{equation*}%\label{Ex1_al1-be1-2}
  \al_1 m - \be_1 (m+1)+ 1 + \frac{\al_1(m+1)}{\nu-1} \le 0, 
\end{equation*}
that is, 
\begin{equation*}%\label{Ex1_pq-2}
  p \le -1 + \frac{q-1}{\nu},
\end{equation*}
then \eqref{XiTh-thm1} and \eqref{ass1} are valid by Lemma \ref{lemm3-Ex1}, and thus Theorem \ref{Thm1} can be applicable. 
\end{example}
%%%%%%%%%%%%%%%%%%

Example \ref{Ex11} tells us that if $(u_0,u_1) \in \cH(\rho_\nu,\infty)$, 
then the dissipative coefficient $b(t)$: 
\begin{equation}\label{b_Ex11}
  b(t) = \mu_0(1+t)^{-1} + (1+t)^{\tp_1} \sin(1+t)^{q},
\end{equation}
where
\begin{equation}
  \tp_1=\tp_1(q,\nu):= p_1 + \frac{q-1}{\nu} = -1 + \frac{q-1}{\nu}
\end{equation}
satisfies the assumptions of Theorem \ref{Thm1} even though $\tilde{p}_1>0$ if $q > \nu+1$, thus the decay estimate \eqref{decay-thm} holds. 
Here, we emphasize that, no results of such an energy decay estimate have been known when $b(t)$ is given by \eqref{b_Ex11} to the best of the author's knowledge.

The following $\si(t)$ is an example that the contribution of $C^m$-property (A3) can be confirmed when Theorem \ref{Thm0} and Theorem \ref{Thm1} are applied. 
%%%%%%%%%%%%%%%%%%%%%%%%%%%%
\begin{example}\label{Ex2}
Let $m$ be a positive integer and $\chi(\tau) \in C^m([0,1])$ satisfy 
\begin{equation*}
  \max_{\tau \in[0,1]}\{|\chi(\tau)|\} = 1, 
  \;\;
  \int^1_0 \chi(\tau)\,d\tau=0 
  \; \text{ and } \;
  |\chi^{(k)}(0)|=|\chi^{(k)}(1)|=0
  \;\; (k=0,\ldots,m). 
\end{equation*}
For $p \ge -1$, $q>1$, $r \ge 1$ and $h \ge 0$, 
we define 
$\{t_n\}_{n=1}^\infty$, $\{N_n\}_{n=1}^\infty$ and $\si_2(t)\in C^m(\R_+)$ by 
\[
  t_n = n^r, 
  \;\;
  N_n = \lfloor n^h \rfloor
\]
and 
\begin{equation}\label{si2}
\si_2(t)=
\begin{cases}
  t_n^p \chi\(t_n^{q-1} (t-t_n) \), 
    & t \in [t_n, t_n + N_n t_n^{-q+1}] \;\;(n=1,2,\ldots),\\[1mm]
  0, 
  & t \in \R_+ \setminus \bigcup_{n=1}^\infty 
  [t_n, t_n + N_n t_n^{-q+1}], 
\end{cases}
\end{equation}
where it must be satisfied for $t_n + N_n t_n^{-q+1} \le t_{n+1}$ that 
\begin{equation*}%\label{h}
  h \le rq - 1.
\end{equation*}
Let $b(t)$ be given by \eqref{mu0-si} with $\si(t)=\si_2(t)$. 
Then, by Lemma \ref{lemm1-Ex2} and Lemma \ref{lemm2-Ex2}, 
there exists a positive constant $C_1$ such that 
\eqref{stb} and \eqref{Cm} are valid 
for $\Th(t)=C_1\Th_2(t)$ and $\varXi(t)=C_1\varXi_1(t)$ with 
\begin{equation}\label{al2}
  \Th_2(t) = (1+t)^{-\al_2}, \;\;
  \al_2 = \al_1 - q + \frac{h+1}{r} = p-2q+2+\frac{h+1}{r}. 
\end{equation}
Here we note that $p$ and $q$ in \eqref{si2} are parameters describing the amplitude and oscillating speed of $\si_2(t)$, respectively, corresponding to those in $\si_1(t)$ defined in \eqref{si-sin}, because \eqref{Cm} is valid for the same $\varXi(t)$ as in Example \ref{Ex1}. 
The new parameters $r$ and $h$ describe the time when the oscillation speed switches and the length of the oscillation intervals, respectively. 
If the following inequality holds: 
\begin{equation*}
  \al_2 m - \be_1 (m+1) + 1 \le 0, 
\end{equation*}
that is, 
\begin{equation*}
  p \le p_2 = p_2(q,r,h,m):=-1 + \frac{m}{m+1}\(q-\frac{h+1}{r}\), 
\end{equation*}
then \eqref{XiTh-thm0} and \eqref{ass0} are valid by Lemma \ref{lemm3-Ex2}, and thus Theorem \ref{Thm0} can be applicable. 
For $\rho(r)=\rho_\nu(r)$, if 
\begin{equation*}
  \al_2 m - \be_1 (m+1)+ 1 + \frac{\al_2(m+1)}{\nu-1} \le 0, 
\end{equation*}
that is, 
\begin{equation*}
\begin{split}
  p \le \tilde{p}_2=\tilde{p}_2(q,r,h,m,\nu):&= 
  -1+\frac{q-1}{\nu} + \(\frac{m}{m+1}+\frac{1}{\nu(m+1)}\)\(q-\frac{h+1}{r}\)
\\
  &=\begin{dcases} 
  \tilde{p}_1 + \(\frac{m}{m+1}+\frac{1}{\nu(m+1)}\)\(q-\frac{h+1}{r}\), 
  \\[1mm]
  p_2 + \frac{q-1}{\nu} + \frac{1}{\nu(m+1)}\(q-\frac{h+1}{r}\),
\end{dcases}
\end{split}
\end{equation*}
then \eqref{XiTh-thm1} and \eqref{ass1} are valid by Lemma \ref{lemm4-Ex2}, and thus Theorem \ref{Thm1} can be applicable. 
\end{example}

The maximum value of $p$ for the amplitude parameter of $\si(t)$ in Example \ref{Ex1}, Example \ref{Ex11} and Example \ref{Ex2} can be summarized as follows: 

\begin{center}
{\renewcommand\arraystretch{2}
\begin{tabular}{|c|c|c|}
\hline
  & Theorem \ref{Thm0} & Theorem \ref{Thm1} 
\\[0.3mm]
\hline
  $\si_1(t)$ & $p_1=-1$ & $\tp_1 = -1 + \frac{q-1}{\nu}$ 
\\[0.7mm]
\hline
  $\si_2(t)$ & $p_2=-1 + \frac{m}{m+1}\(q-\frac{h+1}{r}\) $ & 
  $\tp_2=-1+\frac{q-1}{\nu} + \(\frac{m}{m+1}+\frac{1}{\nu(m+1)}\)\(q-\frac{h+1}{r}\)$
\\[0.7mm]
\hline
\end{tabular}
}
\end{center}

%%%%%%%%%%%%%%%%%%%%%%
\begin{itemize}
\item
If $\si(t)=\si_1(t)$ with $p=\tp_1$ or $\si(t)=\si_2(t)$ with $p=p_2$ or $p=\tp_2$, then the following hold: 
\begin{equation*}
  \limsup_{t\to\infty} 
  \frac{\si(t)}{\mu_0(1+t)^{-1}}=\infty. 
\end{equation*}
Thus $\si(t)$ is no longer small perturbation in the sense of $L^\infty$. 
\item
$\tp_1$, $p_2$ and $\tp_2$ are monotone increasing with respect to $q$, that is, faster oscillation allows for larger amplitudes. 
The contribution of such fast oscillations for the energy decay estimate is also suggested in \cite{GG25} for $\si_1(t)$ with $p=-1$. 
\item 
$\tilde{p}_1$ and $\tilde{p}_2$ are strictly decreasing with respect to $\nu$, and the following hold: 
\[
  \lim_{\nu\to\infty}\tilde{p}_1=p_1
  \;\text{ and }\;
  \lim_{\nu\to\infty}\tilde{p}_2=p_2. 
\]
\item 
If $h<rq-1$, then $p_2$ and $\tilde{p}_2$ are monotone increasing with respect to $m$. 
On the other hand, 
\[
  p_2|_{h=rq-1}=p_1 \;\text{ and }\; \tilde{p}_2|_{h=rq-1}=\tilde{p}_1, 
\]
which implies that the contribution of $m$ cannot be expected if $\si_2(t)$ is oscillating without rest. 
\end{itemize}

Actually, the space of initial data $\cH(\rho_\nu,\infty)$ was introduced to consider the well-posedness of the initial value problem for the wave equation \eqref{uw} with coefficient that has a singularity such as degeneracy and non-Lipschitz continuity (\cite{CDS79, CJS83}). 
Indeed, such singularity of the coefficients may cause loss of regularity of the solution as time evolves, so the initial data should be sufficiently smooth functions, such as functions in the Gevrey class. 
The $L^2$ well-posedness for the Cauchy problem \eqref{u} is trivial since $b(t)\in C^1(\R_+)$, hence it is not necessary to introduce the initial data in $\cH(\rho_\nu,\infty)$ for the well-posedness. 
On the other hand, Theorem \ref{Thm1} is a result showing the contribution of taking initial data in $\cH(\rho_\nu,\infty)$ for the energy decay estimate, and no such result exists as far as the authors know. 
Incidentally, the contribution of the initial data in the Gevrey class for the asymptotic stability of the energy for the Cauchy problem of \eqref{uw} with non-singular coefficient is studied in \cite{EFH15}.  

%%%%%%%%%%%%%%%%%%%%%%%%%%%%%%%%%%%
%
\section{Proof of theorems}
%
%%%%%%%%%%%%%%%%%%%%%%%%%%%%%%%%%%%
%%%%%%%%%%%%%%%%%%%%%%%%%%%%%%%%%%%
\subsection{Outline of the proof}
%%%%%%%%%%%%%%%%%%%%%%%%%%%%%%%%%%%
Before starting the proof of the main theorems, we introduce an outline of the proof and a key proposition. 

By partial Fourier transformation with respect to the spatial variables and denoting 
\[
  v(t,\xi)=\hat{u}(t,\xi),\;\;
  v_0(\xi)=\hat{u}_0(\xi)
  \;\text{ and }\;
  v_1(\xi)=\hat{u}_1(\xi),
\]
\eqref{u} is rewritten as the following Cauchy problem: 
\begin{equation}\label{v}
\begin{dcases}
  \(\pa_t^2 + |\xi|^2 + b(t) \pa_t\)v(t,\xi) = 0, 
  & (t,\xi) \in (0,\infty) \times \R^n,\\
  v(0,\xi) = v_0(\xi),\;\;
  \pa_t v(0,\xi) = v_1(\xi), 
  & \xi \in \R^n. 
\end{dcases}
\end{equation}
By Parseval's theorem, it is sufficient to consider the problem of having a uniform estimate with respect $\xi$ for the following energy density function in the frequency domain: 
\begin{equation*}%\label{cE}
  \cE(t,\xi)=\cE(t,\xi;v):= |\xi|^2 |v(t,\xi)|^2 + |\pa_t v(t,\xi)|^2,\;\;
  (t,\xi)\in [0,\infty)\times \R^n. 
\end{equation*}

Since our interest is the asymptotic behavior as $t \to \infty$, we may assume the following without loss of generality: 
\begin{equation}\label{mu(0)=mu0}
  \mu(0)=\mu_0
  \;\text{ and }\;
  \ze^{-1}(\Th(0))=1. 
\end{equation}
Let $N$ be a large constant. 
We define $t_D=t_D(r)$ and $t_H=t_H(r)$ for $r \in \R_+$ by 
\begin{equation}\label{tDtH}
  t_D(r):=\mu^{-1}\(\frac{\mu_0 r}{N}\)
  \;\;\text{ and }\;\;
  t_H(r):=\Th^{-1}\(\ze\(\frac{r}{N}\)\). 
\end{equation}
Then we divide 
$\R_+ \times \R^n = \{(t,\xi)\;;\; t \in \R_+,\;\; \xi \in \R^n\}$ 
into the following four regions by hypersurfaces 
$t=t_D(|\xi|)$ and $t=t_D(|\xi|)$: 
\[
  Z_D=Z_D(N):=\left\{ (t,\xi)\;;\; 0 \le t \le t_D(|\xi|) \right\},
\]
\[
  Z_H=Z_H(N):=\left\{ (t,\xi)\;;\; 0 \le t \le t_H(|\xi|) \right\},
\]
\[
  Z_{I1}=Z_{I1}(N):=\left\{ (t,\xi)\;;\; t \ge t_D(|\xi|),\; |\xi| \le N\right\}
\]
and
\[
  Z_{I2}=Z_{I2}(N):=
  \left\{ (t,\xi)\;;\; t \ge t_H(|\xi|),\; |\xi| \ge N \right\},
\]
where $Z_D$, $Z_H$, and $Z_I:=Z_{I1} \cup Z_{I2}$ are called 
dissipative zone, hyperbolic zone, and intermediate zone respectively. 

Our main concern in this paper is the effect of the oscillating part $\si(t)$ of the dissipative coefficient on the energy decay estimate. 
The weighted energy method used, for example, in \cite{MN01}, is a powerful method that can be applied to cases where initial boundary value problems with exterior domain and the dissipative coefficients depend on spatial variables, 
but is currently insufficient for precise analysis in models where the dissipative coefficients depend only on $t$.
On the other hand, it is effective to estimate $\cE(t,\xi)$ in each of the divided regions as above when the dissipative coefficient has a different effect on the energy at each region as in our model. 

The following proposition is key to the proof of our main theorems.

%%%%%%%%%%%%%%%%%%
\begin{proposition}\label{Prop1} 
Let us denote 
\begin{equation*}
  \vrho(r):=
  \begin{cases}
    1 & \text{ for Theorem \ref{Thm0}}, \\
    \rho(r) & \text{ for Theorem \ref{Thm1}}. 
  \end{cases}
\end{equation*}
There exist positive constants $N$ and $\ka_0$ such that the following estimates are established: 
\begin{itemize}
\item[(i)] 
For $(t,\xi) \in Z_D(N)$ we have 
\begin{equation*}
  \cE(t,\xi)
  \lesssim 
  \exp\(-2\int^t_0 \mu(\tau)\,d\tau\)
  \(|v_0(\xi)|^2 + |v_1(\xi)|^2 \).
\end{equation*}
\item[(ii)] 
For $(t,\xi) \in Z_H(N)$ we have 
\begin{equation*}
  \cE(t,\xi)
  \lesssim 
  \exp\(-\int^t_0 \mu(\tau)\,d\tau\)
  \exp\(\ka_0 \vrho\(|\xi|\)\)
  \cE(0,\xi).
\end{equation*}
\item[(iii)] 
For $(t,\xi) \in Z_{I1}(N)$ we have 
\begin{equation*}
  \cE(t,\xi)
  \lesssim 
  \exp\(-\int^t_{t_D(|\xi|)} \mu(\tau)\,d\tau\)
  \cE(t_D(|\xi|),\xi).
\end{equation*}
\item[(iv)] 
For $(t,\xi) \in Z_{I2}(N)$ we have 
\begin{equation*}
  \cE(t,\xi)
  \lesssim 
  \exp\(-\int^t_{t_H(|\xi|)} \mu(\tau)\,d\tau\)
  \exp\(\ka_0 \vrho\(|\xi|\)\)
  \cE(t_H(|\xi|),\xi).        
\end{equation*}
\end{itemize}
\end{proposition}
%%%%%%%%%%%%%%%%%%

In the proof of Proposition \ref{Prop1} below, we consider the following first order system, which is rewritten from the second order differential equation of \eqref{v}: 
\begin{equation}\label{V}
  \pa_t V(t,\xi) = A(t,\xi) V(t,\xi),
\end{equation}
where 
\begin{equation}\label{A}
  A(t,\xi):=\begin{pmatrix} 0 & i|\xi| \\ i|\xi| & -b(t) \end{pmatrix}
  \; \text{ and }\;
  V(t,\xi):=\begin{pmatrix} i|\xi|v(t,\xi) \\ \pa_t v(t,\xi) \end{pmatrix}.
\end{equation}
Then the following equality holds: 
\begin{equation}\label{cE=|V|}
  \|V(t,\xi)\|^2_{\C^2} = \cE(t,\xi;v). 
\end{equation}

%%%%%%%%%%%%%%%%%%%%%%%%%%%%%%%%%%%
\subsection{Proof of Proposition \ref{Prop1} (i)} 
%%%%%%%%%%%%%%%%%%%%%%%%%%%%%%%%%%%
We note that the following inequality holds in $Z_D(N)$: 
\begin{equation}\label{txi-ZD}
  |\xi| \le \frac{N}{\mu_0}\mu(t). 
\end{equation}
Since $|\xi| = N\mu_0^{-1}\mu(t_D(|\xi|)) \le N\mu_0^{-1}\mu(0) = N$ in $Z_D(N)$ by (A1), \eqref{mu(0)=mu0} and \eqref{tDtH}, the effect of the oscillation of $\si(t)$ on high frequency energy need not be considered. 
On the other hand, a precise analysis using the representation of the solution of \eqref{v} as a Volterra integral equation is effective and essential in order to derive the decay estimate of $\cE(t,\xi)$ in $Z_D(N)$. 
The method for evaluating $\cE(t,\xi)$ in $Z_D(N)$ follows that used in \cite{HW08}.

We define $\eta_b(t)$, $\eta_\mu(t)$ and $\om(t)$ on $\R_+$ by 
\[
  \eta_b(t):=\exp\(-\int^t_0 b(s)\,ds\), \;\;
  \eta_\mu(t):=\exp\(-\int^t_0 \mu(s)\,ds\)
\]
and
\[
  \om(t):=\exp\(\int_t^\infty \si(s)\,ds\). 
\]
Then we see that 
\begin{equation*}
  \om(t) = \om(0) \exp\(-\int_0^t \si(s)\,ds\)
 =\om(0)\frac{\eta_b(t)}{\eta_\mu(t)}. 
\end{equation*}
Moreover, denoting 
\[
  \om_0:=\inf_{t \ge 0}\{\om(t)\}
  \;\text{ and }\;
  \om_1:=\sup_{t \ge 0}\{\om(t)\}, 
\]
we have
\begin{equation}\label{etamu-etab}
  \frac{\om_0}{\om(0)}\eta_\mu(t)
  \le \eta_b(t) 
  \le \frac{\om_1}{\om(0)}\eta_\mu(t). 
\end{equation}

%%%%%%%%%%%%%%%%%%%%%%%%%%%%%%%%%%%%%%%%
\begin{lemma}\label{lemm-eta|xi|}
The following inequalities hold in $Z_D(N)$: 
\begin{equation}\label{eta|xi|}
  |\xi|\eta_\mu(t)^{-1} < N
\end{equation}
and
\begin{equation}\label{eta|xi|int}
  |\xi|\eta_\mu(t)^{-1}\int^t_0 \eta_\mu(s)\,ds < \frac{N}{1-\mu_0}. 
\end{equation}
\end{lemma}
%%%%%%%%%%%%%%%%%%%%%%%%%%%%%%%%%%%%%%%%
\begin{proof}
By (A1) and \eqref{txi-ZD}, we have 
\begin{align*}
  |\xi|\eta_\mu(t)^{-1}
%=|\xi| \exp\(\int^t_0 \mu(s)\,ds \) 
 \le \frac{N}{\mu_0} \mu(t) \exp\(\int^t_0 \mu_0(1+s)^{-1}\,ds \) 
 = \frac{N}{\mu_0} \mu(t) (1+t)^{\mu_0}
%\\
%&\le \frac{N}{\mu_0}\mu(t) (1+t) 
 < N.
\end{align*}
Moreover, we have 
\begin{align*}
  \eta_\mu(t)^{-1}\int^t_0 \eta_\mu(s)\,ds
& =\int^t_0 \exp\(\int^t_s \mu(\tau)\,d\tau\)\,ds
\\
& \le 
  \int^t_0 \exp\(\int^t_s \mu_0(1+\tau)^{-1}\,d\tau\)\,ds
\\
& =(1+t)^{\mu_0} \int^t_0 (1+s)^{-\mu_0} \,ds
  < \frac{1+t}{1-\mu_0},
\end{align*}
it follows that
\begin{align*}
  |\xi|\eta_\mu(t)^{-1}\int^t_0 \eta_\mu(s)\,ds
& \le \frac{N\mu(t)(1+t)}{\mu_0(1-\mu_0)}
  \le \frac{N}{1-\mu_0}.
\end{align*}
\end{proof}

We define $v_{jk}(t,\xi)$ ($j,k=1,2$) by 
\begin{equation}\label{v1kv1k}
  v_{1k}(t,\xi)
 :=\de_{1k} + i|\xi| \int^t_0 v_{2k}(\tau,\xi)\,d\tau 
\end{equation}
and
\begin{equation}\label{v2kv1k}
  v_{2k}(t,\xi)
 :=\eta_b(t)\(\de_{k2} 
  +i|\xi| \int^t_0 \eta_b(\tau)^{-1} v_{1k}(\tau,\xi)\,d\tau\), 
\end{equation}
where $\de_{jk}=1$ if $j=k$ and $\de_{jk}=0$ if $j \neq k$. 
Then the following lemmas are established.

%%%%%%%%%%%%%%%%%%%%%%%%%%
\begin{lemma}\label{lemm-vjk}
The solution of \eqref{V} is represented in $Z_D(N)$ as follows
\begin{equation*}%\label{V-vjk}
  V(t,\xi)
 =\begin{pmatrix} 
    v_{11}(t,\xi) & v_{12}(t,\xi) \\ v_{21}(t,\xi) & v_{22}(t,\xi)
  \end{pmatrix}
  V(0,\xi). 
\end{equation*}
Moreover, $v_{1k}(t,\xi)$ is a solution of the following integral equation: 
\begin{equation}\label{v11}
\begin{split}
  v_{1k}(t,\xi)
& =\de_{1k} + i \de_{k2} |\xi|\int^t_0\eta_b(\tau_1)\,d\tau_1
\\
&\quad -|\xi|^2 \int^t_0 \eta_b(\tau_1)^{-1}
  \(\int^t_{\tau_1}\eta_b(\tau_2)\,d\tau_2\)
  v_{1k}(\tau_1,\xi)\,d\tau_1
\end{split}
\end{equation}
for $k=1,2$. 
\end{lemma}
%%%%%%%%%%%%%%%%%%%%%%%%%%
\begin{proof}
By \eqref{v1kv1k} and \eqref{v2kv1k} we have 
$v_{1k}(0,\xi)=\de_{1k}$, $v_{2k}(0,\xi)=\de_{2k}$, 
\begin{align*}
  \pa_t v_{1k}
 =i\de_{k2}|\xi|\eta_b(t)
 -|\xi|^2 \eta_b(t) \int^t_0 \eta_b(\tau)^{-1} v_{1k}(\tau,\xi)\,d\tau
 =i|\xi|v_{2k}
\end{align*}
and
\begin{align*}
  \pa_t v_{2k}
& =\eta_b'(t)\(\de_{k2} 
  +i|\xi| \int^t_0 \eta_b(\tau)^{-1} v_{1k}(\tau,\xi)\,d\tau\)
  +i|\xi| v_{1k}(t,\xi)
\\
& =-b(t)v_{2k} + i|\xi| v_{1k}. 
\end{align*}
It follows that
\begin{align*}
  \pa_t \begin{pmatrix} v_{11} & v_{12} \\ v_{21} & v_{22} \end{pmatrix}
  V(0,\xi)
&=\begin{pmatrix}
  i|\xi|v_{21} & i|\xi|v_{22} \\ 
  -b(t)v_{21} + i|\xi| v_{11} & -b(t)v_{22} + i|\xi| v_{12}
  \end{pmatrix}V(0,\xi)
\\
&=\begin{pmatrix} 0 & i|\xi| \\ i|\xi| & -b(t) \end{pmatrix}
  \begin{pmatrix} v_{11} & v_{12} \\ v_{21} & v_{22} \end{pmatrix}V(0,\xi)
\\
&=A(t,\xi)
  \begin{pmatrix} v_{11} & v_{12} \\ v_{21} & v_{22} \end{pmatrix}V(0,\xi).
\end{align*}
\end{proof}

%%%%%%%%%%%%%%%%%%%%%%%%%%
\begin{lemma}\label{Lemm-Volterra}
Let us define 
$w_1(t,\xi)$, $w_2(t,\xi)$, $p_1(t,\xi)$, $p_2(t,\xi)$ and $q(t,\tau,\xi)$ by
\[
  w_1(t,\xi) = |\xi| \eta_b(t)^{-1} v_{11}(t,\xi),\;\;
  w_2(t,\xi) = \eta_b(t)^{-1}v_{12}(t,\xi),
\]
\[
  p_1(t,\xi) = |\xi|\eta_b(t)^{-1} ,\;\;
  p_2(t,\xi) = i|\xi|\eta_b(t)^{-1} \int^t_0\eta_b(\tau_1)\,d\tau_1
\]
and
\[
  q(t,\tau,\xi)=-|\xi|^2\eta_b(t)^{-1} \int^t_{\tau}\eta_b(\tau_1)\,d\tau_1. 
\]
Then $w_k(t,\xi)$ is a solution of the Volterra integral equation \eqref{Vol} with $p(t,\xi)=p_k(t,\xi)$ for $k=1,2$. 
\end{lemma}
%%%%%%%%%%%%%%%%%%%%%%%%%%
\begin{proof}
The proof is trivial by \eqref{v11}. 
\end{proof}

Proposition \ref{Prop1} (i) is immediately proved from 
Lemma \ref{lemm-eta|xi|},  Lemma \ref{lemm-vjk} and the following lemma: 
%%%%%%%%%%%%%%%%%%%%%%%%%%
\begin{lemma}\label{lemma-est_vjk}
The estimates
\begin{equation}\label{lemma-est_vjk-eq}
  |v_{j1}(t,\xi)| \lesssim 1
  \;\text{ and }\;
  |v_{j2}(t,\xi)| \lesssim \eta_\mu(t)
\end{equation} 
hold for $j=1,2$ in $Z_D(N)$. 
\end{lemma}
%%%%%%%%%%%%%%%%%%%%%%%%%%
\begin{proof}
By \eqref{etamu-etab}, we have 
\begin{equation}\label{|p1||p2|}
  |p_1(t,\xi)| \le \frac{\om(0)}{\om_0}|\xi|\eta_\mu(t)^{-1}
  \;\text{ and }\;
  |p_2(t,\xi)| \le \frac{\om_1}{\om_0}|\xi|\eta_\mu(t)^{-1}
    \int^t_0\eta_\mu(\tau)\,d\tau. 
\end{equation}
Suppose that the following inequalities hold: 
\begin{align*}
\int^t_0|q(t,\tau_1,\xi)|\int^{\tau_1}_0|q(\tau_1,\tau_2,\xi)|
  \cdots\int^{\tau_{l-1}}_0|q(\tau_{l-1},\tau_l,\xi)|
  |p_k(\tau_l,\xi)|\,d\tau_l\cdots d\tau_1 
\end{align*}
\begin{empheq}[left = {\le \empheqlbrace \,}]{alignat = 2}
 & \frac{\om(0)}{\om_0} |\xi|\eta_\mu(t)^{-1}
      \frac{1}{(2l)!}\(\sqrt{\frac{\om_1}{\om_0}} |\xi|t \)^{2l} 
 & \quad \text{ \textit{for} $k=1$}, \label{intpqj1} \\
 & \frac{\om_1}{\om_0}|\xi|\eta_\mu(t)^{-1}\int^t_0\eta_\mu(\tau_1)\,d\tau_1
     \frac{1}{(2l)!}
    \(\sqrt{\frac{\om_1}{\om_0}} |\xi|t \)^{2l} 
 & \quad \text{ \textit{for} $k=2$} \label{intpqj2}
\end{empheq}
for any $l =1,2,\ldots$. 
Then, by Lemma \ref{Lemm-Volterra}, Lemma \ref{lemm-Vol} and $|\xi| t \le N$, 
we have 
\begin{align*}
  |w_1(t,\xi)| 
&\le \frac{\om(0)}{\om_0}|\xi|\eta_\mu(t)^{-1}
  + \sum_{j=1}^\infty \frac{\om(0)}{\om_0}|\xi|\eta_\mu(t)^{-1}
      \frac{1}{(2j)!}\(\sqrt{\frac{\om_1}{\om_0}} N\)^{2j}
\\
&\le \frac{\om(0)}{\om_0}\exp\(\sqrt{\frac{\om_1}{\om_0}} N\)
  |\xi| \eta_\mu(t)^{-1}. 
\end{align*}
Analogously, we have 
\begin{align*}
  |w_2(t,\xi)| 
 \le
  \frac{\om_1}{\om_0}\exp\(\sqrt{\frac{\om_1}{\om_0}} N\)
   |\xi| \eta_\mu(t)^{-1}
   \int^{t}_0\eta_\mu(\tau)\,d\tau.
\end{align*}
Therefore, by \eqref{etamu-etab}, \eqref{eta|xi|int}, \eqref{v2kv1k} and Lemma \ref{Lemm-Volterra} and noting that $\eta_\mu(t)$ is monotone decreasing, 
we have
\begin{align*}
  |v_{11}(t,\xi)|
& \le |\xi|^{-1} \eta_b(t)|w_1(t,\xi)|
 \le \frac{\om(0)}{\om_0} \exp\(\sqrt{\frac{\om_1}{\om_0}} N\)
    \frac{\eta_b(t)}{\eta_\mu(t)}
\\
& \le \frac{\om_1}{\om_0} 
   \exp\(\sqrt{\frac{\om_1}{\om_0}} N\), 
\end{align*}
\begin{align*}
  |v_{21}(t,\xi)|
&\le |\xi| \eta_b(t) \int^t_0 \eta_b(\tau)^{-1}|v_{11}(\tau,\xi)|\,d\tau
\\
&\le 
  \(\frac{\om_1}{\om_0}\)^2 \exp\(\sqrt{\frac{\om_1}{\om_0}} N\)
  |\xi| \eta_\mu(t)\int^t_0 \eta_\mu(\tau)^{-1}\,d\tau
\\
&\le\(\frac{\om_1}{\om_0}\)^2 \exp\(\sqrt{\frac{\om_1}{\om_0}} N\)
  |\xi| t
 \le \(\frac{\om_1}{\om_0}\)^2 \exp\(\sqrt{\frac{\om_1}{\om_0}} N\)
  N 
\end{align*}
and
\begin{align*}
  |v_{12}(t,\xi)|
& =\eta_b(t)|w_2(t,\xi)|
\le \frac{\om_1^2}{\om_0\om(0)} \eta_\mu(t)
  \exp\(\sqrt{\frac{\om_1}{\om_0}} N\)
  |\xi| \eta_\mu(t)^{-1}\int^{t}_0\eta_\mu(\tau)\,d\tau
\\
&\le K \eta_\mu(t), 
\end{align*}
where 
\[
  K = \frac{N\om_1^2}{\om_0\om(0)(1-\mu_0)}
  \exp\(\sqrt{\frac{\om_1}{\om_0}} N\). 
\]
Then, by \eqref{v2kv1k} we have 
\begin{align*}
  |v_{22}(t,\xi)|
&\le \eta_b(t)
  \(1 + |\xi|\int^t_0 \eta_b(\tau)^{-1}|v_{12}(\tau,\xi)|\,d\tau\)
\\
&\le \frac{\om_1}{\om(0)} \(1 + \frac{\om_1}{\om(0)} K |\xi|t\) \eta_\mu(t)
\\
&\le \frac{\om_1}{\om(0)} \(1 + \frac{\om_1}{\om(0)} K N\)\eta_\mu(t). 
\end{align*}
Thus we have \eqref{lemma-est_vjk-eq}. 
Let us prove the inequalities \eqref{intpqj1} and \eqref{intpqj2} by induction. 
By \eqref{etamu-etab}, \eqref{|p1||p2|} and 
\[
  |q(t,\tau,\xi)| 
  \le \frac{\om_1}{\om_0}
  |\xi|^2\eta_\mu(t)^{-1}\int^t_{\tau}\eta_\mu(\tau_1)\,d\tau_1, 
\]
furthermore, since $\eta_\mu(t)$ is monotone decreasing, we have 
\begin{align*}
  \int^t_0 |q(t,\tau_1,\xi)| |p_1(\tau_1,\xi)|\,d\tau_1
& \le
  \frac{\om_1 \om(0)}{\om_0^2}
  |\xi|^3 \eta_\mu(t)^{-1} \int^t_0 \eta_\mu(\tau_1)^{-1}\int^t_{\tau_1}
  \eta_\mu(\tau_2)\,d\tau_2 \,d\tau_1
\\
& \le
  \frac{\om_1\om(0)}{\om_0^2}
  |\xi|^3 \eta_\mu(t)^{-1} \int^t_0 \eta_\mu(\tau_1)^{-1}\int^t_{\tau_1}
  \eta_\mu(\tau_1)\,d\tau_2 \,d\tau_1
\\
& =
  \frac{\om_1 \om(0)}{\om_0^2}
  |\xi|^3 \eta_\mu(t)^{-1} \int^t_0 \int^t_{\tau_1} \,d\tau_2 \,d\tau_1
\\
&=
  \frac{\om(0)}{\om_0}
  |\xi| \eta_\mu(t)^{-1} \frac{1}{2!}
  \(\sqrt{\frac{\om_1}{\om_0}}|\xi|t\)^2. 
\end{align*}
Thus \eqref{intpqj1} is valid for $l=1$. 
Suppose that \eqref{intpqj1} is valid for $l \ge 2$. 
Then we have 
\begin{align*}
& \int^t_0|q(t,\tau_1,\xi)|\int^{\tau_1}_0|q(\tau_1,\tau_2,\xi)|
  \cdots\int^{\tau_{l}}_0|q(\tau_{l},\tau_{l+1},\xi)|
  |p_1(\tau_{l+1},\xi)|\,d\tau_{l+1}\cdots d\tau_1 
\\
&\le \int^t_0
  \frac{\om_1}{\om_0}|\xi|^2 \eta_\mu(t)^{-1} 
  \int^t_{\tau_1}\eta_\mu(\tau_2)\,d\tau_2
  \frac{\om(0)}{\om_0} |\xi|\eta_\mu(\tau_1)^{-1}
      \frac{1}{(2l)!}\(\sqrt{\frac{\om_1}{\om_0}} |\xi|\tau_1\)^{2l}
   d\tau_1 
\\
&= \frac{\om(0)}{\om_0}
  \(\sqrt{\frac{\om_1}{\om_0}} |\xi|\)^{2(l+1)}
  |\xi| \eta_\mu(t)^{-1} 
  \frac{1}{(2l)!} \int^t_0
  \eta_\mu(\tau_1)^{-1} \(\int^t_{\tau_1}\eta_\mu(\tau_2)\,d\tau_2\)
       \tau_1^{2l}\,
   d\tau_1 
\\
&\le \frac{\om(0)}{\om_0}
  \(\sqrt{\frac{\om_1}{\om_0}} |\xi|\)^{2(l+1)}
  |\xi| \eta_\mu(t)^{-1} 
  \frac{1}{(2l)!} \int^t_0
  \eta_\mu(\tau_1)^{-1} \(\int^t_{\tau_1}\eta_\mu(\tau_1)\,d\tau_2\)
       \tau_1^{2l}\,
   d\tau_1 
\\
&\le \frac{\om(0)}{\om_0}
  \(\sqrt{\frac{\om_1}{\om_0}} |\xi|\)^{2(l+1)}
  |\xi| \eta_\mu(t)^{-1} 
  \frac{1}{(2l)!} \int^t_0
  \(\int^t_{\tau_1}\,d\tau_2\)
       \tau_1^{2l}\,
   d\tau_1 
\\
&= \frac{\om(0)}{\om_0} |\xi| \eta_\mu(t)^{-1}
  \frac{1}{(2(l+1))!}
  \(\sqrt{\frac{\om_1}{\om_0}} |\xi| t\)^{2(l+1)}. 
\end{align*}
Therefore, \eqref{intpqj1} is valid for any $l \in \N$. 
On the other hand, by the same way as for $k=1$, we have 
\begin{align*}
&  \int^t_0 |q(t,\tau_1,\xi)| |p_2(\tau_1,\xi)|\,d\tau_1
\\
& \le
  \(\frac{\om_1}{\om_0}\)^2
  \eta_\mu(t)^{-1}|\xi|^3 
  \int^t_0 
  \eta_\mu(\tau_1)^{-1}
  \(\int^t_{\tau_1}\eta_\mu(\tau_2)\,d\tau_2\)
  \(\int^{\tau_1}_0\eta_\mu(\tau_2)\,d\tau_2\)
  \,d\tau_1
\\
& \le
  \(\frac{\om_1}{\om_0}\)^2
  |\xi|^3 \eta_\mu(t)^{-1}
  \int^t_0 
  \(\int^t_{\tau_1}\,d\tau_2\)
  \(\int^{\tau_1}_0\eta_\mu(\tau_2)\,d\tau_2\)
  \,d\tau_1
\\
& =
  \(\frac{\om_1}{\om_0}\)^2
  |\xi|^3 \eta_\mu(t)^{-1}
  \int^t_0 \frac{\pa}{\pa \tau_1}\(t\tau_1-\frac12\tau_1^2\) 
  \(\int^{\tau_1}_0\eta_\mu(\tau_2)\,d\tau_2\)
  \,d\tau_1
\\
& =
  \(\frac{\om_1}{\om_0}\)^2
  |\xi|^3 \eta_\mu(t)^{-1}
  \(
    \frac12 t^2 \(\int^{t}_0\eta_\mu(\tau_1)\,d\tau_1\)
   -\int^t_0 \(t\tau_1-\frac12\tau_1^2\) \eta_\mu(\tau_1)\,d\tau_1
  \)
\\
& \le
  \(\frac{\om_1}{\om_0}\)^2
  |\xi|^3 \eta_\mu(t)^{-1}
  \frac12 t^2 \(\int^{t}_0\eta_\mu(\tau_1)\,d\tau_1\)
\\
&= \frac{\om_1}{\om_0} |\xi| \eta_\mu(t)^{-1}
  \int^{t}_0\eta_\mu(\tau_1)\,d\tau_1
  \frac{1}{2!} 
  \(\sqrt{\frac{\om_1}{\om_0}}|\xi| t\)^2. 
\end{align*}
Thus \eqref{intpqj2} is valid for $l=1$. 
Suppose that \eqref{intpqj2} is valid for $l \ge 2$. 
Then we have 
\begin{align*}
& \int^t_0|q(t,\tau_1,\xi)|\int^{\tau_1}_0|q(\tau_1,\tau_2,\xi)|
  \cdots\int^{\tau_{l}}_0|q(\tau_{l},\tau_{l+1},\xi)|
  |p_2(\tau_{l+1},\xi)|\,d\tau_{l+1}\cdots d\tau_1 
\\
&\le
   \int^t_0
   \(\frac{\om_1}{\om_0}
   |\xi|^2 \eta_\mu(t)^{-1}\int^t_{\tau_1}\eta_\mu(\tau_2)\,d\tau_2\)
\\
&\qquad  \times\frac{\om_1}{\om_0}|\xi| \eta_\mu(\tau_1)^{-1}
    \int^{\tau_1}_0\eta_\mu(\tau_2)\,d\tau_2
  \frac{1}{(2l)!}
  \(\sqrt{\frac{\om_1}{\om_0}} |\xi| \tau_1\)^{2l}\, d \tau_1 
\\
&= \frac{\om_1}{\om_0} |\xi| \eta_\mu(t)^{-1}
   \int^t_0
   \eta_\mu(\tau_1)^{-1}\(\int^t_{\tau_1}\eta_\mu(\tau_2)\,d\tau_2\)
   \(\int^{\tau_1}_0\eta_\mu(\tau_2)\,d\tau_2\)
   \tau_1^{2l}\, d \tau_1 
\\
&\qquad \times
   \frac{1}{(2l)!} \(\sqrt{\frac{\om_1}{\om_0}} |\xi| \)^{2l+2}
\\
&\le
\frac{\om_1}{\om_0} |\xi| \eta_\mu(t)^{-1}
   \int^t_0
   \(\int^t_{\tau_1} \,d\tau_2\)
   \(\int^{\tau_1}_0\eta_\mu(\tau_2)\,d\tau_2\)
   \tau_1^{2l}\, d \tau_1 
   \frac{1}{(2l)!} \(\sqrt{\frac{\om_1}{\om_0}} |\xi| \)^{2l+2}
\\
&= \frac{\om_1}{\om_0} |\xi| \eta_\mu(t)^{-1}
   \int^t_0
   \frac{\pa}{\pa \tau_1}
   \(\frac{t}{2l+1}\tau_1^{2l+1} - \frac{1}{2l+2} \tau_1^{2l+2}\)
   \(\int^{\tau_1}_0\eta_\mu(\tau_2)\,d\tau_2\)\, d \tau_1 
\\
&\qquad \times
   \frac{1}{(2l)!} \(\sqrt{\frac{\om_1}{\om_0}} |\xi| \)^{2l+2}
\\
&\le
   \frac{\om_1}{\om_0} |\xi| \eta_\mu(t)^{-1}
   \int^{t}_0\eta_\mu(\tau_1)\,d\tau_1
   \frac{1}{(2(l+1))!} \(\sqrt{\frac{\om_1}{\om_0}} |\xi| t\)^{2(l+1)}.
\end{align*}
Therefore, \eqref{intpqj2} is also valid for any $l \in \N$. 
\end{proof}

\noindent 
\textit{Proof of Proposition \ref{Prop1} (i). }
By \eqref{cE=|V|}, \eqref{eta|xi|} and Lemma \ref{lemma-est_vjk}, we have
\begin{align*}
  \cE(t,\xi)
&=\left\|\begin{pmatrix}
  v_{11}(t,\xi) & v_{12}(t,\xi) \\ v_{21}(t,\xi) & v_{22}(t,\xi)
  \end{pmatrix}
  \begin{pmatrix} i|\xi|v_0(\xi) \\ v_1(\xi) \end{pmatrix}
  \right\|^2_{\C^2}
\\
&=\left|i|\xi| v_{11}(t,\xi) v_0(\xi)+v_{12}(t,\xi)v_1(\xi)\right|^2
 +\left|i|\xi| v_{21}(t,\xi) v_0(\xi)+v_{22}(t,\xi)v_1(\xi)\right|^2
\\
&\le 
  2|\xi|^2\(|v_{11}(t,\xi)|^2 + |v_{21}(t,\xi)|^2\)|v_0(\xi)|^2
\\
&\quad
 +2\(|v_{12}(t,\xi)|^2 + |v_{22}(t,\xi)|^2\) |v_1(\xi)|^2
\\
&\lesssim 
  |\xi|^2 |v_0(\xi)|^2
 +\eta_\mu(t)^2 |v_1(\xi)|^2
\\
&\lesssim 
  \eta_\mu(t)^2\( |v_0(\xi)|^2 + |v_1(\xi)|^2\).
\end{align*}
\qed

%%%%%%%%%%%%%%%%%%%%%%%%%%%%%%%%%%%
\subsection{Proof of Proposition \ref{Prop1} (ii)} 
%%%%%%%%%%%%%%%%%%%%%%%%%%%%%%%%%%%
From now on, $\ze$ denotes 
\begin{equation*}%\label{rzeta}
  \ze(r)=\frac{r}{\vrho(r)}
\end{equation*}
instead of \eqref{zeta}. 
Then \eqref{XiTh-thm0} and \eqref{ass0} can be represented by 
\eqref{XiTh-thm1} and \eqref{ass1}, respectively. 
We note that the following inequality holds in $Z_H(N)$: 
\begin{equation}\label{txi-ZH}
  |\xi| \ge N \ze^{-1}(\Th(t)). 
\end{equation}
In high frequency region $Z_H(N)$, where the effect of the oscillations of $\si(t)$ cannot be neglected, we evaluate $\cE(t,\xi)$ by applying refined diagonalization method, which makes use of the $C^m$-property (A3) introduced in \cite{EFH15, H07, HW08}. 
The following lemma is essential for that method. 
%%%%%%%%%%%%%%%%%%%%%%%%%%
\begin{lemma}\label{Lemm-rd}
Let $\phi, r \in C^1(I)$ for $I \subset \R$, and denote 
$\phi_\Re=\Re\phi$ and $\phi_\Im=\Im\phi$. 
For $t \in I$, we define $\cA=\cA(t)$ and $\cM=\cM(t)$ by 
\begin{equation*}%\label{cA}
  \cA:=\begin{pmatrix} \phi & \ol{r} \\ r & \ol{\phi} \end{pmatrix}
  \;\text{ and }\;
  \cM:=\begin{pmatrix} 1 & \ol{\de} \\ \de & 1 \end{pmatrix},
\end{equation*}
where
\[
  \de=\de(t)
  := -i\frac{r}{\phi_\Im}\(1+\sqrt{1-\(\frac{|r|}{\phi_\Im}\)^2}\)^{-1}. 
\]
If $|\phi_\Im|>|r|$ on $I$, then following equalites hold: 
\begin{equation}\label{Lemm-rd-e1}
  \cM^{-1} \cA \cM = \begin{pmatrix} \la & 0 \\ 0 & \ol{\la} \end{pmatrix},
  \;\;
  \la:= \phi_\Re + i \phi_\Im \sqrt{1-\(|r| / \phi_\Im\)^2}
\end{equation}
and
\begin{equation}\label{Lemm-rd-e2}
\begin{split}
  \cM^{-1} \cA \ \frac{d}{dt}\cM
& =\frac{d}{dt} \log \sqrt{1-|\de|^2}
    \begin{pmatrix} 1 & 0 \\ 0 & 1 \end{pmatrix}
\\
&\quad +\frac{1}{1-|\de|^2}
  \begin{pmatrix} 
    i\Im\(\ol{\de} \de'\) & \ol{\de'} \\ 
    \de' & \ol{i\Im\(\ol{\de} \de'\)} \end{pmatrix}. 
\end{split}
\end{equation}
\end{lemma}
%%%%%%%%%%%%%%%%%%%%%%%%%%
\begin{proof}
Note that $\det\cM = 1-|\de|^2>0$, 
we can verify \eqref{Lemm-rd-e1} and \eqref{Lemm-rd-e2} by direct calculations.  
Clearly, $\la$ and $\ol{\la}$ are eigenvalues of $\cA$, and $\cM$ is the diagonalizer of $\cA$. 
\end{proof}

For $A=A(t,\xi)$ defined by \eqref{A}, 
we define $M$, $A_0=A_0(t,\xi)$ and $M_0=M_0(t,\xi)$ by 
\begin{equation}\label{M}
  M: = \begin{pmatrix} 1 & -1 \\ 1 & 1 \end{pmatrix}, 
\end{equation}
\begin{equation}\label{A0}
  A_0(t,\xi)
 :=M^{-1} A M
 =\begin{pmatrix} \phi_0 & r_0 \\ r_0 & \ol{\phi_0}\end{pmatrix}
\end{equation}
and 
\begin{equation*}
  M_0(t,\xi):=\begin{pmatrix}
    1 & \ol{\de_0} \\ \de_0 & 1
  \end{pmatrix},
\end{equation*}
where 
\begin{equation*}
  \phi_0=\phi_0(t,\xi):=i|\xi|-\frac{b(t)}{2},\;\;
  r_0=r_0(t):=-\frac{b(t)}{2}
\end{equation*}
and
\begin{equation*}
  \de_0=\de_0(t,\xi)
:=-i\frac{b(t)}{2|\xi|}\(1+\sqrt{1-\(\frac{b(t)}{2|\xi|}\)^2}\)^{-1}. 
\end{equation*}
By \eqref{XiTh-thm1}, there exists a positive constant $b_1$ such that 
\begin{equation*}
  b_1=\sup_{t\ge0} \left\{\frac{\varXi(t)}{\ze^{-1}\(\Th(t)\)}\right\}. 
\end{equation*}
If $N > b_1/2$, then we have 
\begin{equation}\label{r0phi0}
  |r_0(t)|
  \le \frac{\varXi(t)}{2}
  \le \frac{b_1 \ze^{-1}\(\Th(t)\)}{2}
  \le \frac{b_1 |\xi|}{2N}
  < |\xi| = |\Im \phi_0(t,\xi)|
\end{equation}
by \eqref{Cm} and \eqref{txi-ZH}. 
Therefore, by considering $\xi$ as a parameter, 
setting $A_0 = \cA$ and $\cM=M_0$, 
and applying Lemma \ref{Lemm-rd}, we have the following equality: 
\begin{equation*}
  M_0^{-1} A_0 M_0 - M_0^{-1}A_0 \(\pa_t M_0\)
 =:A_1(t,\xi)
 =\begin{pmatrix} 
    \phi_{1} & \ol{r_{1}} \\ r_{1} & \ol{\phi_{1}}
  \end{pmatrix},
\end{equation*}
where 
\begin{equation*}
  \la_0=\la_0(t,\xi)
  :=-\frac{b(t)}{2} + i|\xi|\sqrt{1-\(\frac{b(t)}{2|\xi|}\)^2}, 
\end{equation*}
\begin{equation*}
  \phi_1=\phi_1(t,\xi):
 =\la_0 - \pa_t \log \sqrt{1-|\de_0|^2}
 -\frac{i\Im\(\ol{\de_0} \pa_t \de_0\)}{1-|\de_0|^2}
\end{equation*}
and
\begin{equation*}
  r_1=r_1(t,\xi) =-\frac{\pa_t \de_0}{1-|\de_0|^2}. 
\end{equation*}
If $N \ge b_1$, then \eqref{Cm}, \eqref{txi-ZH} and \eqref{r0phi0} imply the followings: 
\begin{align*}
  |\de_0| \le \frac{|b(t)|}{2|\xi|} 
  \le \frac{\varXi(t)}{2|\xi|}
  \le \frac{b_1}{2N}
  \le \frac{1}{2}, 
\end{align*}
\begin{align*}
  |\pa_t \de_0|
&\le
  \frac{|b'(t)|}{2|\xi|}
 +\frac{|b(t)|^2|b'(t)|}{8|\xi|^3}
  \(1-\(\frac{b(t)}{2|\xi|}\)^2\)^{-\frac12}
\\
&\le
  \frac{|b'(t)|}{2|\xi|}
 +\frac{|b(t)|^2|b'(t)|}{8|\xi|^3}
  \(\frac{3}{4}\)^{-\frac12}
\\
&\le
  \frac{|b'(t)|}{2|\xi|} \(1 + \(\frac{b(t)}{2|\xi|}\)^2\)  
 \le \frac{3\varXi(t)^2}{4|\xi|}. 
\end{align*}
Therefore, we have 
\begin{equation*}
  |r_1| \le \frac{1}{1-\frac{1}{4}}\frac{3\varXi(t)^2}{4|\xi|}
  \le \frac{\varXi(t)^2}{|\xi|}
\end{equation*}
and
\begin{equation*}
  \phi_{1\Re}=-\frac{b(t)}{2}-\pa_t \log\sqrt{1-|\de_0|}. 
\end{equation*}
Thus we can again apply Lemma \ref{Lemm-rd} as $\cA = A_1$ by taking $N$ large enough. 
By repeating this procedure, we expect the following lemma to hold.

%%%%%%%%%%%%%%%%%%%%%%%%%%%%%%%%%%%%%%%%%%%
\begin{lemma}\label{Lemma-Cm} 
Let $A_0$ be defined by \eqref{A0}, and 
$M_k=M_k(t,\xi)$, $\de_k=\de_k(t,\xi)$ and $A_{k+1}=A_{k+1}(t,\xi)$ for $k=0,\ldots,m-1$ be defined inductively as follows: 
\[
  M_k:=\begin{pmatrix} 1 & \ol{\de_k} \\ \de_k & 1 \end{pmatrix},
  \quad
  \de_k:=-i\frac{r_k}{\phi_{k\Im}}
  \(1+\sqrt{1-\(\frac{|r_k|}{\phi_{k\Im}}\)^2}\)^{-1}
\]
and
\[
  A_{k+1}
  =\begin{pmatrix} 
    \phi_{k+1} & \ol{r_{k+1}} \\ r_{k+1} & \ol{\phi_{k+1}}
  \end{pmatrix}
  :=M_k^{-1} A_k M_k - M_k^{-1} \(\pa_t M_k\).
\]
Then, there exists a constant $N_0 \ge 1$ such that 
\[
  \sup_{(t,\xi)\in Z_H(N)}\{|\de_k(t,\xi)|\} \le \frac12
\]
for any $k=0,\ldots,m-1$ and $N \ge N_0$. 
Furthermore, the following properties hold: 
\begin{equation}\label{phi_mR}
  \phi_{m\Re}(t,\xi)
 =-\frac{b(t)}{2}-\frac12\pa_t \log \prod_{k=0}^{m-1} \(1-|\de_k|^2\) 
\end{equation}
and
\begin{equation}\label{r_k}
  |r_k(t,\xi)| \lesssim \frac{\varXi(t)^{k+1}}{|\xi|^k}
\end{equation}
for $k=0,\ldots,m$. 
\end{lemma}
%%%%%%%%%%%%%%%%%%%%%%%%%%%%%%%%%%%%%%%%%%%

We have already confirmed that \eqref{r_k} is correct for $k=0$ and $k=1$ by representing them concretely, however, it is difficult to do it for general $k\ge2$. 
Therefore, we shall show \eqref{r_k} by introducing the following symbol classes that can describe the orders of $\varXi(t)$ and $|\xi|$. 

%%%%%%%%%%%%%%%%%%%%%%%%%%%%
\begin{definition}
For $N \ge 1$, 
integers $p$ and $q$, and non-negative integer $k$, we define the symbol class $S^{(k)}\{p,q\}$ by the following class of functions: 
\begin{equation*}%\label{sef-symbol}
  S^{(k)}\{p,q\}:=
  \left\{
  f(t,\xi) \in C^k(Z_H(N))\;;\;
  \max_{0 \le j \le k}\sup_{(t,\xi)\in Z_H(N)}\left\{
  \frac{|\pa_t^j f(t,\xi)|}{|\xi|^{p} \varXi(t)^{q+j}}\right\}
  <\infty
  \right\}.
\end{equation*}
\end{definition}
%%%%%%%%%%%%%%%%%%%%%%%%%%%%

As properties of the symbol class $S^{(k)}\{p,q\}$, we introduce the following lemmas: 
%%%%%%%%%%%%%%%%%%%%%%%%%%%%%%%
\begin{lemma}\label{lemm-symbol1}
\begin{itemize}
\item[(i)] 
If $f \in S^{(k)}\{p,q\}$, then $f \in S^{(j)}\{p+r,q-r\}$ for any $0 \le j \le k$ and $r \ge 0$. 
\item[(ii)] 
If $f_1,f_2 \in S^{(k)}\{p,q\}$, then for any constants  $C_1$ and $C_2$ we have $C_1 f_1 + C_2 f_2 \in S^{(k)}\{p,q\}$. 
\item[(iii)] 
If $f_1 \in S^{(k_1)}\{p_1,q_1\}$ and $f_2 \in S^{(k_2)}\{p_2,q_2\}$, then 
$f_1 f_2 \in S^{(\min\{k_1,k_2\})}\{p_1+p_2,q_1+q_2\}$. 
\item[(iv)] 
If $f \in S^{(k)}\{p,q\}$ and $k \ge 1$, then $\pa_t f \in S^{(k-1)}\{p,q+1\}$. 
\end{itemize}
\end{lemma}
%%%%%%%%%%%%%%%%%%%%%%%%%%%%%%%
\begin{proof} 
Let us prove only (i); the others are trivial from the definition of the symbol classes. 
If $f \in S^{(k)}\{p,q\}$, then $f \in S^{(j)}\{p,q\}$ for $0 \le j \le k$ is trivial. 
Since $\varXi(t) \le b_1N^{-1} |\xi|$ in $Z_H(N)$, we have 
\begin{align*}
  |\pa_t^l f(t,\xi)| 
  \lesssim |\xi|^p \varXi(t)^{q+l}
  \le \(\frac{b_1}{N}\)^r
  |\xi|^{p+r} \varXi(t)^{q-r+l}
\end{align*}
for $0\le l \le j$, 
it follows that $f \in S^{(k)}\{p+r,q-r\}$. 
\end{proof}

%%%%%%%%%%%%%%%%%%%%%%%%%%%%%%%
\begin{lemma}\label{lemm-symbol2}
There exists a positive constant $N_0$ such that the following properties hold for any $N \ge N_0$: 
\begin{itemize}
\item[(i)] 
For any $\ve>0$, if $f \in S^{(k)}\{-r,r\}$ with $r \ge 1$, then 
$\sup_{(t,\xi)\in Z_H(N)}\{|f(t,\xi)|\} < \ve$. 
\item[(ii)] 
If $f \in S^{(k)}\{0,0\}$ and $\inf_{(t,\xi)\in Z_H(N)}\{ |f(t,\xi)|\}>0$, then  $1/f(t,\xi) \in S^{(k)}\{0,0\}$. 
\item[(iii)] 
If $f \in S^{(k)}\{1,0\}$ and 
$\inf_{(t,\xi)\in Z_H(N)}\{ |f(t,\xi)|/|\xi|\}>0$, then 
$1/f(t,\xi) \in S^{(k)}\{-1,0\}$. 
\item[(iv)] 
If $f \in S^{(k)}\{-p,p\}$ with $p\ge 1$, then $\sqrt{1-|f|^2} \in S^{(k)}\{0,0\}$. 
\end{itemize}
\end{lemma}
%%%%%%%%%%%%%%%%%%%%%%%%%%%%%%%
\begin{proof}
(i) is confirmed by setting $p=q=0$ and taking $N$ large enough with respect to $\ve$ in the proof of Lemma \ref{lemm-symbol1} (i). 
Since (ii) is proved by the same way as the proof of (iii), we shall prove (iii). 
By applying Fa\`a di Bruno's formula:
\begin{equation}\label{FBF}
  \frac{d^j}{dt^j}F(G(x))
 =j!\sum_{h=1}^j 
  F^{(h)}(G(t))
  \sum_{\substack{h_1 + 2h_2 + \cdots + n h_n=j \\ 
                  h_1+h_2+\cdots + h_j = h}}
  \,
  \prod_{l=1}^j \frac{1}{h_l!\, l!^{h_l}}\(G^{(l)}(t)\)^{h_l} 
\end{equation}
as $F(G)=1/G$ and $G=f(t,\xi)$, and note that 
$|F^{(h)}(G)|\lesssim |G|^{-h-1}$, 
we have 
\begin{align*}
  \left|\pa_t^j \frac{1}{f(t,\xi)}\right|
\lesssim\:&
  \sum_{h=1}^j \frac{1}{f(t,\xi)^{h+1}}
  \sum_{\substack{h_1 + 2h_2 + \cdots + n h_n=j \\ 
                  h_1+h_2+\cdots + h_k = h}}
  \,
  \prod_{l=1}^j \left|\pa_t^l f(t,\xi)\right|^{h_l}
\\
\lesssim\:&
  \sum_{h=1}^k |\xi|^{-h-1}
  \sum_{\substack{h_1 + 2h_2 + \cdots + n h_n=j \\ 
                  h_1+h_2+\cdots + h_j = h}} \,
  \prod_{l=1}^j \(|\xi| \varXi(t)^{-l}\)^{h_l}
\\
\lesssim\:&
  |\xi|^{-1} \varXi(t)^{-j}
\end{align*}
for $j=0,\ldots,k$. 
We shall prove (iv) by applying \eqref{FBF} as $F(G)=\sqrt{1-G}$ and $G=|f(t,\xi)|^2$. 
We can assume that $|f(t,\xi)|^2 \le 1/2$ by (i). 
Since 
\[
  |F^{(h)}(G)| \lesssim \frac{F(G)}{(1-G)^h} 
  \le \frac{1}{(1-G)^h} 
  \le 2^h
\]
and $|f(t,\xi)|^2 \in S^{(k)}\{-2p,2p\}$ by Lemma \ref{lemm-symbol1} (iii), 
we have 
\begin{align*}
  \left|\pa_t^j \sqrt{1-|f(t,\xi)|^2}\right|
&\lesssim
  \sum_{h=1}^j 2^h 
  \sum_{\substack{h_1 + 2h_2 + \cdots + n h_n=j \\ 
                  h_1+h_2+\cdots + h_j = h}}
  \,
  \prod_{l=1}^j \left|\pa_t^l |f(t,\xi)|^2\right|^{h_l}
\\
&\lesssim
  \sum_{h=1}^j 
  \sum_{\substack{h_1 + 2h_2 + \cdots + n h_n=j \\ 
                  h_1+h_2+\cdots + h_j = h}}
  \,
  \prod_{l=1}^j \(|\xi|^{-2q}\varXi(t)^{-2q-l}\)^{h_l}
\\
&\lesssim
  \sum_{h=1}^j |\xi|^{-2qh}\varXi(t)^{-2qh-j}
\le \varXi(t)^{-j} \sum_{h=1}^j \(\frac{b_1}{N}\)^{2qh}
\\
&\lesssim \varXi(t)^{-j} 
\end{align*}
for $j=0,\ldots,k$. 
Thus (iv) is proved.  
\end{proof}

\noindent 
\textit{Proof of Lemma \ref{Lemma-Cm}. } 
Let $N$ be set enough large so that Lemma \ref{lemm-symbol2} can be applied. 
First of all, the following properties are clear: 
\[
  \phi_{0\Re} \in S^{(m)}\{0,1\}, \;\;
  \phi_{0\Im} \in S^{(m)}\{1,0\}, \;\;
  r_0 \in S^{(m)}\{0,1\}
  \;\text{ and }\;
  1 \in S^{(m)}\{0,0\}. 
\]
Since $b(t)|\xi|^{-1} \in S^{(m)}\{-1,1\}$ by Lemma \ref{lemm-symbol1} (iii), 
Lemma \ref{lemm-symbol2} (iv) implies
\[
  \sqrt{1-\(\frac{b(t)}{2|\xi|}\)^2} \in S^{(m)}\{0,0\}. 
\]
Therefore, we have 
\[
  \de_0 \in S^{(m)}\{-1,1\} 
\]
by Lemma \ref{lemm-symbol2} (ii). 
Moreover, Lemma \ref{lemm-symbol2} (i) and (ii) imply 
\[
  \frac{1}{1-|\de_0|^2} \in S^{(m)}\{0,0\}. 
\]
Consequently, we have 
\[
  r_1 = -\frac{\pa_t \de_0}{1-|\de_0|^2} \in S^{(m-1)}\{-1,2\}. 
\]
Now we suppose that the following properties are valid: 
\begin{equation*}%\label{rk-phik}
  r_k \in S^{(m-k)}\{-k,k+1\}, \;\;
  \phi_{k\Im} \in S^{(m-k)}\{1,0\}
  \;\text{ and }\;
  \frac{1}{\phi_{k\Im}} \in S^{(m-k)}\{-1,0\} 
\end{equation*}
for $1 \le k \le m-1$.  
Then we have 
\begin{equation}\label{phikI}
  \inf_{(t,\xi)\in Z_H(N)}\left\{\frac{|\phi_{k\Im}(t,\xi)|}{|\xi|}\right\}>0, 
\end{equation}
and
\begin{equation}\label{rkphikI}
  \frac{r_k}{\phi_{k\Im}} \in S^{(m-k)}\{-k-1,k+1\}
  \;\text{ and }\;
  \frac{|r_k|}{|\xi|} \ll 1
\end{equation}
by Lemma \ref{lemm-symbol1} (iii) and Lemma \ref{lemm-symbol2} (i). 
Therefore, we can apply Lemma \ref{Lemm-rd} as $\phi=\phi_k$ and $r=r_k$. 
By the same argument as for $k=0$, we have 
\begin{equation*}%\label{symbol-dek}
  \de_k \in S^{(m-k)}\{-k-1,k+1\}
  \;\text{ and }\;
  \frac{1}{1-|\de_k|^2} \in S^{(m-k)}\{0,0\}, 
\end{equation*}
and thus
\[
  r_{k+1} = -\frac{\pa_t \de_k}{1-|\de_k|^2} \in S^{(m-k-1)}\{-k-1,k+2\}. 
\]
It follows that 
\[
  \phi_{(k+1)\Im}
 =\phi_{k\Im}\sqrt{1-\(\frac{|r_k|}{\phi_{k\Im}}\)^2}
   -\Im\(\ol{\de_k} r_{k+1}\)
  \in S^{(m-k-1)}\{1,0\}
\]
by Lemma \ref{Lemm-rd}. 
Furthermore, by \eqref{phikI}, 
$\ol{\de_k}r_{k+1}|\xi|^{-1} \in S^{(m-k-1)}\{-2k-3,2k+3\}$ 
and Lemma \ref{lemm-symbol2} (i), we have 
\[
  \inf_{(t,\xi)\in Z_H(N)}
  \left\{\frac{|\phi_{(k+1)\Im}|}{|\xi|}\right\}>0. 
\]
Therefore, by Lemma \ref{lemm-symbol2} (iii) we obtain
\[
  \frac{1}{\phi_{(k+1)\Im}} \in S^{m-k-1}\{-1,0\}. 
\]
This procedure of applying Lemma \ref{Lemm-rd} can be repeated until $k=m$.
Moreover, \eqref{phi_mR} follows from Lemma \ref{Lemm-rd}. 
\qed

%%%%%%%%%%%%%%%%%%%%%%%%%%%%%%%%%%%%%%%%%%%%%%%
\medskip
\noindent 
\textit{Proof of Proposition \ref{Prop1} (ii).} 
Let $N_0$ be a sufficiently large number such that $|\de_k| \le 1/2$ for $k=0,\ldots,m-1$ by applying Lemma \ref{lemm-symbol2} (i) as $\de_k \in S^{(m-k)}\{-k-1,k+1\}$ and $\ve \le 1/2$. 
Then we have
\begin{equation*}
  \sup_{(t,\xi)\in Z_H(N)}\{\|M_k(t,\xi)\|_{\C^{2 \times 2}}\} \le 1 
  \;\text{ and }\;
  \sup_{(t,\xi)\in Z_H(N)}\{\|M_k(t,\xi)^{-1}\|_{\C^{2 \times 2}}\} 
  \le \frac{4}{3} 
\end{equation*}
for $k=0,\ldots,m-1$, it follows that 
\begin{equation*}
  \|\cM Y\|^2_{\C^2} \le 4\|Y\|^2_{\C^2}
  \;\text{ and }\;
  \|\cM^{-1} Y\|^2_{\C^2} 
  \le 4\(\frac43\)^2\|Y\|^2_{\C^2}
  \le 6\|Y\|^2_{\C^2} 
\end{equation*}
for any $Y \in \C^2$ and $\cM = M, M_0, \ldots,M_{m-1}$. 
For the solution $V$ of \eqref{V}, we define $V_j=V_j(t,\xi)$ ($j=1,\ldots,m$) by 
\[
  V_j(t,\xi):=M_{j-1}(t,\xi)^{-1} \cdots M_0(t,\xi)^{-1} M^{-1}V(t,\xi). 
\]
Here we note that the following equalities hold: 
\begin{align*}
  \pa_t V_1
&=\pa_t M_0^{-1}M^{-1}V
 =-M_0^{-1}\(\pa_t M_0\) M_0^{-1} M^{-1}V + M_0^{-1}M^{-1} \pa_t V
\\
&=\(A_1 - M_0^{-1}A_0M_0\)M_0^{-1} M^{-1}V + M_0^{-1}M^{-1} A V
\\
&=A_1 V_1 - M_0^{-1} \(M^{-1}AM\) M^{-1}V + M_0^{-1}M^{-1} A V
\\
&=A_1 V_1
\end{align*}
and
\begin{align*}
  \pa_t V_{k+1}
&=\(\pa_t M_k^{-1}\) V_{k} + M_{k}^{-1} \pa_t V_{k}
 =-M_{k}^{-1} \(\pa_t M_{k}\) M_{k}^{-1} V_{k} 
  + M_{k}^{-1} \pa_t V_{k}
\\
&=\(A_{k+1} - M_k^{-1} A_k M_k\) M_{k}^{-1} V_{k} 
  + M_{k}^{-1} \pa_t V_{k}
\\
&=A_{k+1}V_{k+1} + M_{k}^{-1} \(\pa_t V_{k} - A_k  V_{k}\)
\end{align*}
for $k=1,\ldots,m-1$, 
it follows that 
\begin{equation*}
  \pa_t V_m(t,\xi) = A_m(t,\xi) V_m(t,\xi). 
\end{equation*}
Therefore, we have 
\begin{align*}
  \pa_t \|V_m\|_{\C^2}^2
&=2\Re\(\pa_t V_m,V_m\)_{\C^2}
 =2\Re\(A_m V_m,V_m\)_{\C^2}
\\
&=2\phi_{m\Re} \|V_m\|_{\C^2}^2 
  +2\Re\(\begin{pmatrix} 0 & \ol{r_m}\\ r_m & 0\end{pmatrix} V_m, V_m\)_{\C^2}
\\
&\le \(2\phi_{m\Re} + 2|r_m|\)\|V_m\|_{\C^2}^2,
\end{align*}
and thus 
\begin{align*}
  \|V_m(t,\xi)\|_{\C^2}^2
  \le \exp\(\int^t_0 2\phi_{m\Re}(s,\xi)\,ds\)
      \exp\(\int^t_0 2|r_m(s,\xi)|\,ds\) \|V_m(0,\xi)\|_{\C^2}^2
\end{align*}
for $0 \le t \le t_H(|\xi|)$ by Gronwall's inequality. 
By \eqref{ass1}, \eqref{tDtH} and \eqref{r_k}, we have 
\begin{align*}
  \int^t_0 2|r_m(s,\xi)|\,ds 
&\lesssim |\xi|^{-m}\int^{t_H}_0 \varXi(s)^{m+1} \,ds
  \lesssim \frac{\ze^{-1}(\Th(t_H))^{m+1}}{|\xi|^{m} \Th(t_H)}
 = \frac{\ze^{-1}\(\Th(t_H)\)}{N^m \Th(t_H)}
\\
&= N^{-m} \frac{\ze^{-1}\(\ze\(\frac{|\xi|}{N}\)\)}{\ze\(\frac{|\xi|}{N}\)}
 = N^{-m}\frac{\frac{|\xi|}{N}}{\ze\(\frac{|\xi|}{N}\)}
\\
&\lesssim \vrho\(|\xi|\). 
\end{align*}
Moreover, by \eqref{phi_mR} and $|\de_k|\le 1/2$, we have 
\begin{align*}
  \int^t_0 2\phi_{m\Re}(s,\xi)\,ds
& =-\int^t_0 b(s)\,ds 
  -\log\prod_{k=0}^{m-1}\frac{1-|\de_k(t,\xi)|}{1-|\de_k(0,\xi)|}
\\
& \le -\int^t_0 b(s)\,ds +\log 2^m. 
\end{align*}
Summarizing the estimates above and \eqref{etamu-etab}, there exists a positive constant $\ka_0$ such that 
\begin{align*}
  \|V_m(t,\xi)\|_{\C^2}^2
&\le 2^m \exp\(-\int^t_0 b(s)\,ds\)
     \exp\(\ka_0 \vrho\(|\xi|\)\) \|V_m(0,\xi)\|_{\C^2}^2
\\
&\le \frac{2^m \om_1}{\om(0)} \exp\(-\int^t_0 \mu(\tau)\,d\tau\)
     \exp\(\ka_0 \vrho\(|\xi|\)\) \|V_m(0,\xi)\|_{\C^2}^2. 
\end{align*}
Finally, by \eqref{cE=|V|} and the inequalites 
\begin{align*}
  \|V_m(t,\xi)\|^2_{\C^2}
 =\left\|M_{m-1}^{-1} \cdots M_0^{-1} M^{-1} V(t,\xi)\right\|^2_{\C^2}
 \le 6^{m+1}\|V(t,\xi)\|^2_{\C^2}
\end{align*}
and
\begin{align*}
  \|V(t,\xi)\|^2_{\C^2}
 =\left\|M M_0 \cdots M_{m-1} V_m(t,\xi)\right\|^2_{\C^2}
 \le 4^{m+1}\|V_m(t,\xi)\|^2_{\C^2},
\end{align*}
we conclude the proof of Proposition \ref{Prop1} (ii).
\qed
%%%%%%%%%%%%%%%%%%%%%%%%%%

%%%%%%%%%%%%%%%%%%%%%%%%%%%%%%%%%%%
\subsection{Proof of Proposition \ref{Prop1} (iii) and (iv)} 
%%%%%%%%%%%%%%%%%%%%%%%%%%%%%%%%%%%
We note that the following inequalites hold in $Z_I(N)$: 
\begin{equation}\label{zone_ZI}
  \frac{N}{\mu_0}\mu(t) \le |\xi| \le N
  \;\text{ or }\;
  N \le |\xi| \le N\ze^{-1}\(\Th(t)\). 
\end{equation}
The properties given in the next lemma are essential in the proof of Proposition \ref{Prop1} (iii) and (iv). 
%%%%%%%%%%%%%%%%%%%%%%%%%%%%%%%%%
\begin{lemma}\label{lemm-estZI}
Let $A=A(t,\xi)$ be defined by \eqref{A}. 
For any $N \ge 1$, there exist a matrix $\tM = \tM(t,\xi)$ which is uniformly regular in $Z_I(N)$, and a diagonal matrix $\tLa=\tLa(t,\xi)$ such that the following properties hold: 
\begin{itemize}
\item[(i)] 
$\tLa$ is represented by $\tLa=\diag(\tla,\ol{\tla})$ and $\tla=\tla(t,\xi)$ satisfies 
\begin{equation}\label{lemm-estZI-eq1}
  \Re \tla(t,\xi) = -\frac{\mu(t)}{2}. 
\end{equation}
\item[(ii)] 
The matirix $\tR=\tR(t,\xi)$ defined by 
\begin{equation}\label{def_tR}
  \tR(t,\xi):=\tM^{-1} A \tM - \tLa - \tM^{-1}\(\pa_t \tM\)
\end{equation}
satisfies
\begin{equation}\label{lemm-estZI-eq2}
  \int^\infty_{t_D(|\xi|)} \|\tR(s,\xi)\|_{\C^{2 \times 2}}\,ds
  \lesssim 1
\end{equation}
and
\begin{equation}\label{lemm-estZI-eq3}
  \int^\infty_{t_H(|\xi|)} \|\tR(s,\xi)\|_{\C^{2 \times 2}}\,ds
  \lesssim \vrho\(|\xi|\). 
\end{equation}
\end{itemize} 
\end{lemma}
%%%%%%%%%%%%%%%%%%%%%%%%%%%%%%%%%

If Lemma \ref{lemm-estZI} is valid, then denoting 
\begin{equation*}
  \tV=\tV(t,\xi):=\tM(t,\xi)^{-1} V(t,\xi),
\end{equation*}
we have 
\begin{align*}
  \pa_t \tV
&=\tM^{-1} \pa_t V - \tM^{-1} \(\pa_t \tM\) \tM^{-1} V
\\
&=\tM^{-1} A \tM \tV - \tM^{-1} \(\pa_t \tM\) \tV
 =\(\tLa + \tR\) \tV. 
\end{align*}
Thus, $\tM$ can be thought of as a matrix that performs diagonalization of the principal part of $A$ in $Z_I(N)$, 
that is, the effect of $\si(t)$ is eliminatied from the diagonal components by $\tM$.
In the proof of Lemma \ref{lemm-estZI} below, we construct $\tM$ by a three-step procedure. 

\smallskip
\noindent 
\textit{Proof of Lemma \ref{lemm-estZI}}. 
We define 
$\tM_1=\tM_1(t)$ and $\tA_1=\tA_1(t,\xi)$ by 
\begin{equation*}
  \tM_1(t):=\begin{pmatrix} 1 & 0 \\ 0 & \om(t) \end{pmatrix} 
\end{equation*}
and
\begin{equation*}
  \tA_1(t,\xi):=\tM_1^{-1}A\tM_1 - \tM_1^{-1} \(\pa_t \tM_1\). 
\end{equation*}
Then we have 
\begin{align*}
  \tA_1(t,\xi)
&=\begin{pmatrix} 1 & 0 \\ 0 & \om(t)^{-1} \end{pmatrix}
  \begin{pmatrix} 0 & i|\xi|\om(t) \\ i|\xi| & -b(t)\om(t) \end{pmatrix}
 -\begin{pmatrix} 1 & 0 \\ 0 & \om(t)^{-1} \end{pmatrix}
  \begin{pmatrix} 0 & 0 \\ 0 & -\si(t)\om(t) \end{pmatrix}
\\
&=\begin{pmatrix} 1 & 0 \\ 0 & \om(t)^{-1} \end{pmatrix}
  \begin{pmatrix} 0 & i|\xi|\om(t) \\ i|\xi| & -\mu(t)\om(t) \end{pmatrix}
\\
&=\begin{pmatrix} 
    0 & i|\xi|\om(t) \\ i|\xi|\om(t)^{-1} & -\mu(t)
  \end{pmatrix}. 
\end{align*}

In the next step, we define $\tA_2=\tA_2(t,\xi)$ by 
\begin{equation*}
  \tA_2(t,\xi):=M^{-1} \tA_1 M, 
\end{equation*}
where $M$ was defined by \eqref{M}. 
Let us define $\tPhi_2=\tPhi_2(t,\xi)$, $\tR_2=\tR_2(t,\xi)$ and $\tr_{\pm}=\tr_{\pm}(t,\xi)$ by 
\[
  \tPhi_2(t,\xi)
 :=\begin{pmatrix}
    -\frac{\mu(t)}{2}+i|\xi| & -\frac{\mu(t)}{2} \\
    -\frac{\mu(t)}{2} & -\frac{\mu(t)}{2}-i|\xi|
  \end{pmatrix},
\]
\begin{align*}
  \tR_2(t,\xi):=
  \begin{pmatrix}
    \tr_+(t,\xi) & \ol{\tr_-(t,\xi)} \\ 
    \tr_-(t,\xi) & \ol{\tr_+(t,\xi)} 
  \end{pmatrix}
  \;\text{ and }\;
  \tr_{\pm}(t,\xi):
 =\frac{i|\xi|(1-\om(t))\(\om(t)^{-1} \mp 1\)}{2},
\end{align*}
and note the following equalites:
\begin{align*}
  \om(t)^{-1} + \om(t) - 2 = (1-\om(t))\(\om(t)^{-1} - 1\)
\end{align*}
and
\begin{align*}
  \om(t)^{-1} - \om(t) = (1-\om(t))\(\om(t)^{-1} + 1\). 
\end{align*}
Then we have 
\begin{align*}
  \tA_2
&=\frac12\begin{pmatrix}
    - \mu(t) +i|\xi|\om(t)^{-1} + i|\xi|\om(t) 
  & - \mu(t) -i|\xi|\om(t)^{-1} + i|\xi|\om(t) \\
    - \mu(t) +i|\xi|\om(t)^{-1} -i|\xi|\om(t) 
  & - \mu(t) -i|\xi|\om(t)^{-1} - i|\xi|\om(t) 
  \end{pmatrix}
\\
&=\tPhi_2(t,\xi) + \tR_2(t,\xi).
\end{align*}
The symbol class of $\tr_{\pm}$ is $S^{(m)}\{1,0\}$ in $Z_H(N)$ hence it seems to be crucial for the energy estimate. 
However, later considerations will show us that they do not actually affect the estimates of $\cE(t,\xi)$ in $Z_I(N)$ by (A2) and \eqref{stb}.
The real part of the diagonal components $-\mu(t)/2$ provides the decay order we expect. 
On the other hand, $-\mu(t)/2$ in the off-diagonal components may cancel out the decay obtained from the diagonal components. 

In the next step, we perform further diagonalization procedure so that the off-diagonal components do not affect the decay estimate in $Z_I(N)$ by similar way as in $Z_H(N)$. 
An eigenvalue and the corresponding eigenvector of $\tPhi_2$ are given by 
\begin{equation}\label{tla}
  \tla=\tla(t,\xi)
  =-\frac{\mu(t)}{2} + i|\xi|\sqrt{1-\(\frac{\mu(t)}{2|\xi|}\)^2}
\end{equation}
and ${}^t(1, \tde)$, respectively, where 
\begin{equation*}
  \tde=\tde(t,\xi)
  =i\frac{\mu(t)}{2|\xi|}
  \(1+\sqrt{1-\(\frac{\mu(t)}{2|\xi|}\)^2}\)^{-1}. 
\end{equation*}
Since the other eigenvalue and the corresponding eigenvector are 
$\ol{\tla}$ and ${}^t(\ol{\tde},1)$, respectively, and 
\begin{equation}\label{|tde|}
\begin{split}
  |\tde(t,\xi)|
&=\frac{\mu(t)}{2|\xi|}
  \(1+\sqrt{1-\(\frac{\mu(t)}{2|\xi|}\)^2}\)^{-1}
\le
  \frac{1}{2N}
  \frac{1}{1+\sqrt{1-\(\frac{1}{2N}\)^2}}
\\
&\le \frac{1}{2N} \le \frac12
\end{split}
\end{equation}
in $Z_I(N)$ for $N\ge 1$, by the diagonalizer $\tM_2=\tM_2(t,\xi)$ defined by 
\begin{equation*}
  \tM_2(t,\xi)
  :=\begin{pmatrix} 1 & \ol{\tde} \\ \tde & 1 \end{pmatrix},
\end{equation*}
we have 
\[
  \tM_2^{-1} \tPhi_2 \tM_2 
 =\tLa:=\begin{pmatrix} \tla & 0 \\ 0 & \ol{\tla} 
  \end{pmatrix}. 
\]

In the final step, we prove the estimates \eqref{lemm-estZI-eq2} and \eqref{lemm-estZI-eq3}. 
Let us define $\tM=\tM(t,\xi)$ by 
\[
  \tM:=\tM_1 M \tM_2. 
\]
It is clear that $\tM$ is regular, and furthermore, 
we see that 
\begin{equation}\label{bdd-tM}
  \|\tM\|_{\C^{2 \times 2}} \lesssim 1
  \;\text{ and }\;
  \|\tM^{-1}\|_{\C^{2 \times 2}} \lesssim 1
\end{equation}
in $Z_I(N)$ by noting the following inequalities: 
\begin{equation*}
  \|\tM_1\|_{\C^{2 \times 2}} \le \om_1, \;\;
  \|\tM_1^{-1}\|_{\C^{2 \times 2}} \le \om_0^{-1}, \;\;
  \|\tM_2\|_{\C^{2 \times 2}} \le 1 
  \;\text{ and }\;
  \|\tM_2^{-1}\|_{\C^{2 \times 2}} \le \frac{4}{3}. 
\end{equation*}
Let us define $\tR=\tR(t,\xi)$ by 
\begin{equation}\label{tR}
  \tR:=
  \tM_2^{-1} \tR_2 \tM_2 -\tM_2^{-1} \(\pa_t \tM_2\). 
\end{equation}
Then we can confirm that $\tR$ satisfies \eqref{def_tR} as follows: 
\begin{align*}
  \tR
%&=\tM_2^{-1} \tR_2 \tM_2 - \tM_2^{-1} \(\pa_t \tM_2\)\\
&=\tM_2^{-1} \(\tA_2 - \tPhi_2\) \tM_2 
 +\tM_2^{-1} M^{-1} \tM_1^{-1} 
  \(\(\pa_t \tM_1\) M \tM_2-\(\pa_t \tM_1 M \tM_2\)\)
\\
&=\tM_2^{-1} M^{-1} \tA_1 M \tM_2 - \tLa 
 +\tM^{-1} \(\pa_t \tM_1\) M \tM_2
 -\tM^{-1} \(\pa_t \tM\)
\\
&=\tM_2^{-1} M^{-1} \(\tM_1^{-1}A\tM_1 
  - \tM_1^{-1} \(\pa_t \tM_1\)\) M \tM_2 
\\
&\quad +\tM^{-1} \(\pa_t \tM_1\) M \tM_2
 - \tLa 
 -\tM^{-1} \(\pa_t \tM\)
\\
&=\tM^{-1} A \tM - \tLa -\tM^{-1} \(\pa_t \tM\), 
\end{align*}
moreover, \eqref{lemm-estZI-eq1} is valid by \eqref{tla}. 
Finally, we shall prove \eqref{lemm-estZI-eq2} and \eqref{lemm-estZI-eq3}. 
Note that 
\begin{align*}
  \tM_2^{-1} \(\pa_t \tM_2\)
 =\frac{1}{1-|\tde|^2}
    \begin{pmatrix} 
      -\ol{\tde} \pa_t \tde & \pa_t \ol{\tde}\\
      \pa_t \tde & -\tde\pa_t \ol{\tde}
    \end{pmatrix},
\end{align*}
\begin{align*}
  \sup_{(t,\xi)\in Z_I(N)}
    \left\{\int^\infty_t |\xi|^{-1}|\mu'(s)| \,ds\right\}
  \le 
  \begin{cases}
    N^{-1} \mu_0 &\text{ in $Z_{I1}(N)$}, \\
    N^{-1} \mu(0) &\text{ in $Z_{I2}(N)$},
  \end{cases}
\end{align*}
\eqref{|tde|}, \eqref{bdd-tM} and the following estimates: 
\begin{align*}
  |\pa_t \tde|
\lesssim 
   \frac{|\mu'(t)|}{|\xi|}
 +\frac{|\mu'(t)|\mu(t)^2}{|\xi|^3}
\lesssim \frac{|\mu'(t)|}{|\xi|},
\end{align*}
we have 
\begin{align*}
&  \sup_{(t,\xi)\in Z_I} \left\{
  \int^\infty_t 
  \left\|\tM_2(s,\xi)^{-1} \(\pa_t \tM_2(s,\xi)\)\right\|_{\C^{2 \times 2}}
   \,ds\right\}
\\
&  \lesssim 
  \sup_{(t,\xi)\in Z_I} 
  \left\{ \int^\infty_t |\xi|^{-1} |\mu'(s)|\,ds\right\}
  \lesssim 1. 
\end{align*}
By mean value theorem, there exists $\th \in (0,1)$ such that 
\begin{align*}
  \left|\om(t)-1\right|
& =\exp\(\th \int^\infty_t \si(\tau)\,d\tau\)
    \left|\int^\infty_t \si(\tau)\,d\tau\right|
\\
&\le \om_1^\th \left| \int^\infty_t \si(\tau)\,d\tau\right|
\le \om_1 \left| \int^\infty_t \si(\tau)\,d\tau\right|.
\end{align*}
Moreovre, note that 
\begin{align*}
  \left|\om(t)^{-1} \pm 1\right| \le \om_0^{-1} + 1 \le 2\om_0^{-1}
\end{align*}
by \eqref{stb} and \eqref{zone_ZI}
we have 
\begin{align*}
  \sup_{t \ge t_D(|\xi|)}
  \left\{\int^\infty_t \left|\tilde{r}_{\pm}(s,\xi)\right| \,ds\right\}
  \le
  \frac{\om_1}{\om_0}|\xi|
  \int^\infty_{t_D(|\xi|)}\left|\int^\infty_s \si(\tau)\,d\tau \right| \,ds
  \le \frac{N \om_1 }{\om_0 \Th(0)}
\end{align*}
and
\begin{align*}
  \sup_{t \ge t_H(|\xi|)}
  \left\{\int^\infty_t \left|\tilde{r}_{\pm}(s,\xi)\right| \,ds\right\}
&\le
 \frac{\om_1}{\om_0}|\xi|\int^\infty_{t_H(|\xi|)} 
    \left|\int^\infty_s \si(\tau)\,d\tau \right| \,ds
 \le
 \frac{\om_1}{\om_0}\frac{|\xi|}{\Th(t_H(|\xi|))}
\\
&=\frac{\om_1}{\om_0}\frac{|\xi|}{\ze\(\frac{|\xi|}{N}\)}
 =\frac{N \om_1}{\om_0} \vrho\(\frac{|\xi|}{N}\)
 \le \frac{N \om_1}{\om_0} \vrho(|\xi|).
\end{align*}
Therefore, by \eqref{bdd-tM} and \eqref{tR}, 
we have \eqref{lemm-estZI-eq2} and \eqref{lemm-estZI-eq3}. 
\qed

\smallskip
\noindent
\textit{Proof of Proposition \ref{Prop1} (iii) and (iv).} 
Let us define $\tV=\tV(t,\xi)$ by 
\begin{equation*}
  \tV(t,\xi):=\tM(t,\xi)^{-1} V(t,\xi). 
\end{equation*}
Then  the following equalties hold: 
\begin{align*}
  \pa_t \tV
&=\tM^{-1} \pa_t V - \tM^{-1} \(\pa_t \tM\) \tM^{-1} V(t,\xi)
\\
&=\tM^{-1} A \tM \tV - \tM^{-1} \(\pa_t \tM\) \tV(t,\xi)
=\(\tLa + \tR\) \tV. 
\end{align*}
By the same way as in the proof of Proposition \ref{Prop1} (ii) and Lemma \ref{lemm-estZI}, we have 
\begin{align*}
  \pa_t \|\tV(t,\xi)\|_{\C^2}^2
&=2\Re\(\pa_t \tV(t,\xi), \tV(t,\xi)\)_{\C^2}
\\
&=-\mu(t) \|\tV(t,\xi)\|_{\C^2}^2
 +2\Re\(\tR \tV(t,\xi), \tV(t,\xi)\)_{\C^2}
\\
& \le \(-\mu(t) + 2\|\tR(t,\xi)\|_{\C^{2 \times 2}}\) \|\tV(t,\xi)\|_{\C^2}^2, 
\end{align*}
it follows that 
\begin{align*}
  \|\tV(t,\xi)\|_{\C^2}^2 
  \lesssim 
  \begin{dcases}
    \exp\(-\int^t_{t_D(|\xi|)} \mu(s) \,ds\) \|\tV(t_D(|\xi|),\xi)\|_{\C^2}^2 
    & \text{\textit{in} $Z_{I1}(N)$},
    \\
    \exp\(-\int^t_{t_H(|\xi|)} \mu(s) \,ds\) \exp\(\ka_0\vrho\(|\xi|\)\)
    \|\tV(t_H(|\xi|),\xi)\|_{\C^2}^2 
    & \text{\textit{in} $Z_{I2}(N)$} 
  \end{dcases}
\end{align*}
for a positive constant $\ka_0$. 
Finally, since $\tM$ and $\tM^{-1}$ are uniformly bounded in $Z_I(N)$, 
and thus 
\begin{equation*}
  \|\tV(t,\xi)\|_{\C^2}^2 \simeq \|V(t,\xi)\|_{\C^2}^2 = \cE(t,\xi)
\end{equation*}
hold, we conclude the proof of Proposition \ref{Prop1} (iii) and (iv). 
\qed

%%%%%%%%%%%%%%%%%%%%%%%%%%%%%%%%%%%
\subsection{Proof of the main theorems} 
%%%%%%%%%%%%%%%%%%%%%%%%%%%%%%%%%%%

For $t \in \R_+$ we define $\Om_j(t) \subset \R^n$ and $I_j(t)$ ($j=1,2,3,4$) by 
\begin{align*}
 &\Om_1(t) := \left\{\xi \in \R^n\;;\; |\xi|\le N\mu_0^{-1}\mu(t) \right\}, 
\\
 &\Om_2(t) := \left\{\xi \in \R^n\;;\; N\mu_0^{-1}\mu(t) \le |\xi|\le N
 \right\}, 
\\
 &\Om_3(t) := \left\{\xi \in \R^n\;;\; N \le |\xi| \le N\ze^{-1}(\Th(t)) \right\}, 
\\
 &\Om_4(t) := \left\{\xi \in \R^n\;;\; |\xi| \ge N\ze^{-1}(\Th(t)) \right\}, 
\end{align*}
and
\begin{align*}
  I_j(t) := \int_{\Om_j(t)} \cE(t,\xi) \, d\xi.
\end{align*}
Here we note that the following equalities hold 
\[
  E(t;u)=\int_{\R^n}\cE(t,\xi;v)\,d\xi
 =\sum_{j=1}^4 I_j(t). 
\]
Since $t \le t_D(|\xi|)$ for $\xi \in \Om_1(t)$, by Proposition \ref{Prop1} (i) and Parseval's theorem, we have 
\begin{align*}
  I_1(t)
 &\lesssim
  \exp\(-2\int^t_0 \mu(\tau)\,d\tau\)
  \int_{\Om_1(t)} \(|v_0(\xi)|^2 + |v_1(\xi)|^2 \) \, d\xi
\\
 &\le 
  \exp\(-2\int^t_0 \mu(\tau)\,d\tau\)\(\|u_0(\cd)\|^2+\|u_1(\cd)\|^2\).
\end{align*}
Analogously, since 
$t \ge t_D(|\xi|)$, $t \ge t_H(|\xi|)$, and $t \le t_H(|\xi|)$ 
for $\xi \in \Om_2(t)$, $\xi \in \Om_3(t)$, and $\xi \in \Om_4(t)$, 
respectively, we have the following estimates by Proposition \ref{Prop1}: 
\begin{align*}
  I_2(t)
&\lesssim
  \int_{\Om_2(t)} \exp\(-\int^t_{t_D} 
  \mu(\tau)\,d\tau\)\cE(t_D,\xi)\,d\xi
\\
&\lesssim
  \int_{\Om_2(t)} 
  \exp\(-\int^t_{t_D} 
  \mu(\tau)\,d\tau-2\int^{t_D}_0 \mu(\tau)\,d\tau\)
  \(|v_0(\xi)|^2 + |v_1(\xi)|^2 \)
  \,d\xi
\\
&\le
  \exp\(-\int^t_{0} \mu(\tau)\,d\tau\)\(\|u_0(\cd)\|^2+\|u_1(\cd)\|^2\),
\end{align*}
\begin{align*}
  I_3(t)
&\lesssim
  \int_{\Om_3(t)} \exp\(-\int^t_{t_H} \mu(\tau)\,d\tau\)
  \exp\(\ka_0 \vrho\(|\xi|\)\)
  \cE(t_H,\xi)\,d\xi
\\
&\lesssim
  \exp\(-\int^t_0 \mu(\tau)\,d\tau\)
  \int_{\Om_3(t)} 
  \exp\(2\ka_0 \(|\xi|\)\)
  \cE(0,\xi)
  \,d\xi
\\
&\lesssim
  \begin{dcases}
    \exp\(-\int^t_0 \mu(\tau)\,d\tau\) \(\|\nabla u_0(\cd)\|^2+\|u_1(\cd)\|^2\)
      & \text{\textit{if $\vrho(r)=1$}},\\
    \exp\(-\int^t_0 \mu(\tau)\,d\tau\) \|(u_0,u_1)\|^2_{\cH(\rho,\ka_0)}
      & \text{\textit{if $\vrho(r)=\rho(r)$}},
  \end{dcases}
\end{align*}
and
\begin{align*}
  I_4(t)
&\lesssim
  \exp\(-\int^t_0 \mu(\tau)\,d\tau\) \int_{\Om_4(t)} 
  \exp\(\ka_0\vrho\(|\xi|\)\)
  \cE(0,\xi)\,d\xi
\\
&\lesssim
  \begin{dcases}
    \exp\(-\int^t_0 \mu(\tau)\,d\tau\) \(\|\nabla u_0(\cd)\|^2+\|u_1(\cd)\|^2\)
      & \text{\textit{if $\vrho(r)=1$}},\\
    \exp\(-\int^t_0 \mu(\tau)\,d\tau\) \|(u_0,u_1)\|^2_{\cH(\rho,\ka_0)}
      & \text{\textit{if $\vrho(r)=\rho(r)$}}. 
  \end{dcases}
\end{align*}
The above estimates and Lemma \ref{Lemm-cH} (i) conclude the proof of Theorem \ref{Thm0} and Theorem \ref{Thm1}. 
\qed

%%%%%%%%%%%%%%%%%%%%%%%%%%%%%%%%%%%
%
\section{Appendix}
%
%%%%%%%%%%%%%%%%%%%%%%%%%%%%%%%%%%%

%%%%%%%%%%%%%%%%%%%%%%%%%%%%%%%%%%%
\subsection{Verification of Example \ref{Ex1} and Example \ref{Ex11}}
%%%%%%%%%%%%%%%%%%%%%%%%%%%%%%%%%%%
Let the parameters $p$ and $q$ satisfy 
\[
  p \ge -1, \;\;
  q>1
  \;\text{ and }\;
  p < q-2, 
\]
and the dippative coefficient $b(t)$ be defined in Example \ref{Ex1} as follows: 
\[
  b(t)=\mu_0(1+t)^{-1} + \si_1(t)
  \;\text{ and }\;
  \si_1(t)= (1+t)^p \sin (1+t)^q. 
\]
Then we have the following lemma: 

%%%%%%%%%%%%%%%%%%%%%%%%%%%%%%%
\begin{lemma}\label{lemm1-Ex1}
The following estimate holds:
\begin{equation}\label{est_intsi1}
  \int^\infty_t 
  \left|\int^\infty_s \si_1(\tau) \,d\tau\right|
  \,ds
  \lesssim (1+t)^{p-q+2}. 
\end{equation}
\end{lemma}
%%%%%%%%%%%%%%%%%%%%%%%%%%%%%%%
\begin{proof}
We note that the following well-known estimate holds for any $\nu>0$: 
\begin{equation}\label{est-HN}
  \left|\int^\infty_s (1+\th)^{-\nu} \sin(1+\th)\,d\th\right|
  \lesssim (1+s)^{-\nu}. 
\end{equation}
For a proof, see e.g. \cite{HN18}. 
By \eqref{est-HN} with $\nu = (p-q+1)/q$, we have 
\begin{align*}
  \int^\infty_s (1+\tau)^{p} \sin(1+\tau)^q \,d\tau
&=\frac{1}{q}\int^\infty_{(1+t)^q-1} 
   (1+\th)^{\frac{p-q+1}{q}} \sin(1+\th) \,d\th
\\
&\lesssim
  (1+t)^{p-q+1},  
\end{align*}
it follows that \eqref{est_intsi1} since $p-q+1<-1$. 
\end{proof}

%%%%%%%%%%%%%%%%%%%%%%%%%%%%%%%
\begin{lemma}\label{lemm2-Ex1}
For any positive integer $m$, the following estimates hold: 
\begin{equation*}
  \left|\frac{d^k}{dt^k} b(t)\right|
  \lesssim (1+t)^{-\be_1(k+1)}, 
  \;\; k=0,\ldots,m. 
\end{equation*}
\end{lemma}
%%%%%%%%%%%%%%%%%%%%%%%%%%%%%%%
\begin{proof}
We note that 
\begin{align*}
  \be_1
 =-q+1+\frac{-p+q-1}{m+1}
  < -q + 1 + (-p + q - 1)
%  = -p 
  \le 1. 
\end{align*}
Let $0 \le k \le m$. 
By Leibniz rule, we have 
\begin{align*}
  \left|\frac{d^k}{dt^k}\si_1(t)\right|
 &=\left|\sum_{j=0}^k \binom{k}{j}
    \(\frac{d^j}{dt^{j}}(1+t)^{p}\)
    \(\frac{d^{k-j}}{dt^{{k-j}}}\sin(1+t)^{q}\) \right|
\\
&\lesssim \sum_{j=0}^k (1+t)^{p-j+(k-j)(q-1)}
  =\sum_{j=0}^k (1+t)^{p+(k-j)q-k}
  \lesssim (1+t)^{p+k(q-1)}
\\
&=(1+t)^{-(k+1)\(-q+1+\frac{-p+q-1}{k+1}\)}
  \le (1+t)^{-(k+1)\(-q+1+\frac{-p+q-1}{m+1}\)}
\\
&  = (1+t)^{-\be_1(k+1)}.
\end{align*}
Moreover, the following estimates are trivial: 
\[
  \left|\frac{d^k}{dt^k}\mu_0(1+t)^{-1}\right|
  \lesssim (1+t)^{-(k+1)}
  \le (1+t)^{-\be_1(k+1)}. 
\]
\end{proof}

%%%%%%%%%%%%%%%%%%%%%%%%
\begin{lemma}\label{lemm3-Ex1}
If $p \le \tilde{p}_1(q,\nu)$, then \eqref{XiTh-thm1} and \eqref{ass1} are vaild. 
\end{lemma}
\begin{proof}
If $p \le \tilde{p}_1(q,\nu) = -1 + (q-1)/\nu$, that is, 
\[
  p+1+\frac{p-q+2}{\nu-1} \le 0 
\]
holds, then we have 
\begin{align*}
  \al_1 m - \be_1 (m+1)+ 1 + \frac{\al_1(m+1)}{\nu-1}
 =\(p+1+\frac{p-q+2}{\nu-1}\)(m+1)
 \le 0. 
\end{align*}
It follows that \eqref{XiTh-thm1} and \eqref{ass1} are valid by Example \ref{Ex0} (ii) with $\al=\al_1$ and $\be=\be_1$. 
\end{proof}

%%%%%%%%%%%%%%%%%%%%%%%%%%%%%%%%%%%
\subsection{Verification of Example \ref{Ex2}}
%%%%%%%%%%%%%%%%%%%%%%%%%%%%%%%%%%%
Let the dissipative coefficient $b(t)$ be defined in Example \ref{Ex2} as follows: 
\[
  b(t):=\mu_0(1+t)^{-1} + \si_2(t), 
\]
where $\si_2(t)$ is defined by \eqref{si2}, that is, 
\[
\si_2(t)=
\begin{cases}
  t_n^p \chi\(t_n^{q-1} (t-t_n) \), 
    & t \in [t_n, t_n + N_n t_n^{-q+1}] \;\; (n=1,2,\ldots),\\[1mm]
  0, 
  & t \in \R_+ \setminus \bigcup_{n=1}^\infty 
  [t_n, t_n + N_n t_n^{-q+1}]. 
\end{cases}
\]
Here we note that 
\[
  N_n t_n^{-q+1} = \lfloor n^h \rfloor n^{-r(q-1)}
  \simeq n^{- r(q-1) + h }
  \simeq t_n^{- q +1 + \frac{h}{r}},
\]
and we assumed that 
\begin{equation*}
  p \ge -1, \;\;
  q>1, \;\;
  r \ge 1, \;\;
  h \ge 0,
\end{equation*}
\begin{equation*}
  h \le rq - 1
  \;\text{ and }\;
  p\le -1 + \frac{m}{m+1}\(q-\frac{h+1}{r}\). 
\end{equation*}
Then we see that 
\begin{equation}\label{pqrh}
  p < q - 1 -\frac{h+1}{r} \le 
  \begin{dcases}
    q - 1,\\
    2q - 2 - \frac{h+1}{r}.
  \end{dcases}
\end{equation}
We shall show \eqref{al2}, that the following lemma holds: 
%%%%%%%%%%%%%%%%%%%%%%%%%%%%%%%%%
\begin{lemma}\label{lemm1-Ex2}
The following estimate holds:  
\begin{equation*}
  \int^\infty_t \left|\int^\infty_s \si_2(\tau)\,d\tau\right|\,ds
  \lesssim (1+t)^{p - 2q + 2 + \frac{h+1}{r}}.
\end{equation*}
\end{lemma}
%%%%%%%%%%%%%%%%%%%%%%%%%%%%%%%%%
\begin{proof}
Let $s\in[t_k,t_{k+1})$. 
Here we note that 
\begin{equation}\label{t-simeq-hk}
  s \le \frac{t_{k+1}}{t_k} t_k
  =\(\frac{k+1}{k}\)^r t_k \le 2^r t_k
\end{equation}
holds. 
If $s \ge t_k + N_k t_{k}^{-q+1}$, then 
$\int^{t_{k+1}}_{s} \si_2(\tau)\,d\tau = 0$. 
On the other hand, if $s < t_k + N_k t_{k}^{-q+1}$, then there exists 
$l \in \N$ such that 
\[
  t_k + (l-1) t_{k}^{-q+1} \le s < t_k + l t_{k}^{-q+1}
\]
and 
\begin{align*}
  \int^{t_{k+1}}_{s} \si_2(\tau)\,d\tau
&=\int^{t_k + l t_{k}^{-q+1}}_{s} \si_2(\tau)\,d\tau
  +\sum_{j = l+1}^{N_k}
   \int^{t_k + j t_{k}^{-q+1}}_{t_k + (j-1) t_{k}^{-q+1}} \si_2(\tau)\,d\tau
\\
&=\int^{t_k + l t_{k}^{-q+1}}_{s} \si_2(\tau)\,d\tau,
\end{align*}
it follows that 
\begin{align*}
  \left|\int^{t_{k+1}}_{s} \si_2(\tau)\,d\tau\right|
  \le\int^{t_k + l t_{n}^{-q+1}}_{s}|\si_2(\tau)|\,d\tau
  \le\int^{t_k + l t_{n}^{-q+1}}_{t_k+(l-1)t_k^{-q+1}}t_k^p\,d\tau
  \le t_{k}^{p-q+1}. 
\end{align*}
Hence we have 
\begin{align*}
  \left|\int^{\infty}_{s} \si_2(\tau)\,d\tau\right|
  \le 
  \begin{cases}
    t_{k}^{p-q+1}, 
      & s \in \bigcup_{k=1}^\infty 
        \(t_k, t_k + N_k t_{k}^{-q+1}\),\\[2mm]
    0, & s \in \bigcup_{k=1}^\infty 
        \left[t_k + N_k t_{k}^{-q+1}, t_{k+1}\right].
  \end{cases}
\end{align*}
For any fixed $t \in \R_+$, we set $n \in \N$ satisfying 
$t_n \le t < t_{n+1}$. 
Since 
\[
  r(p-2q+2) + h = -1 - r\(-p+2q-2 - \frac{h+1}{r}\) < -1 
\]
by \eqref{pqrh}, we have 
\begin{align*}
  \int^\infty_t \left|\int^\infty_s \si_2(\tau)\,d\tau\right|\,ds
&\le \sum_{k=n}^\infty
  \int^{t_{k+1}}_{t_k} 
  \left|\int^\infty_s \si_2(\tau)\,d\tau\right|\,ds
\\
&\le \sum_{k=n}^\infty \int^{t_k + N_k t_{k}^{-q+1}}_{t_k} t_k^{p-q+1} \,ds
 =\sum_{k=n}^\infty N_k t_{k}^{p-2q+2}
\\
& \le \sum_{k=n}^\infty k^{r(p-2q+2) + h}
  \lesssim n^{r(p-2q+2) + h+1}
\\
&
= t_n^{p-2q+2 + \frac{h+1}{r}}
\le \(2^{-r} t\)^{p-2q+2 + \frac{h+1}{r}}
\\
&\simeq (1+t)^{p-2q+2 + \frac{h+1}{r}}. 
\end{align*}
\end{proof}

%%%%%%%%%%%%%%%%%%%%%%%%%%%
\begin{lemma}\label{lemm2-Ex2}
The following estimates hold: 
\begin{equation*}
  \left|\frac{d^k}{dt^k} \si_2(t)\right|
  \lesssim (1+t)^{(k+1)\(-q + 1 + \frac{-p+q-1}{m+1}\)},
  \;\; k=0,\ldots,m. 
\end{equation*}
\end{lemma}
%%%%%%%%%%%%%%%%%%%%%%%%%%%
\begin{proof}
Let $t \in [t_n,t_{n+1})$. 
By the same way to the proof of Lemma \ref{lemm2-Ex1} and using \eqref{t-simeq-hk}, we have 
\begin{align*}
  \left|\frac{d^k}{dt^k}\si_2(t)\right|
\le
  t_n^{p+k(q-1)}\max_{\tau\in[0,1]}\left\{\left|\chi^{(k)}(\tau)\right|\right\}
\lesssim (1+t)^{p+k(q-1)}
\le (1+t)^{-\be_1(k+1)}
\end{align*}
for $k=0,\ldots,m$. 
\end{proof}

%%%%%%%%%%%%%%%%%%%%%%%%%%%%%%%%%%
\begin{lemma}\label{lemm3-Ex2}
If $p \le p_2(q,r,h,m)$, then \eqref{XiTh-thm0} and \eqref{ass0} are vaild. 
\end{lemma}
\begin{proof}
If $p \le p_2(q,r,h,m)$, that is, 
\[
  p + 1 - \frac{m}{m+1}\(q-\frac{h+1}{r}\) \le 0
\]
holds, then we have 
\begin{align*}
  \al_2 m - \be_1 (m+1)+ 1
&=\al_1 m - \be_1 (m+1)+ 1 - \(\al_1-\al_2\)m
\\
& =(p+1)(m+1) - \(q-\frac{h+1}{r}\)m
\\
&=\(p+1 - \frac{m}{m+1}\(q-\frac{h+1}{r}\)\)(m+1)
 \le 0. 
\end{align*}
It follows that \eqref{XiTh-thm0} and \eqref{ass0} are valid by Example \ref{Ex0} (i) with $\al=\al_2$ and $\be=\be_1$. 
\end{proof}

%%%%%%%%%%%%%%%%%%%%%%%%%%%%%%%%%%
\begin{lemma}\label{lemm4-Ex2}
If $p \le \tp_2(q,r,h,m,\nu)$, then \eqref{XiTh-thm1} and \eqref{ass1} are vaild. 
\end{lemma}
\begin{proof}
If $p \le \tp_2(q,r,h,m,\nu)$, that is, 
\[
  p + 1 - \frac{q-1}{\nu} 
  - \(\frac{m}{m+1}+\frac{1}{\nu(m+1)}\)\(q-\frac{h+1}{r}\) \le 0
\]
holds, then we have 
\begin{align*}
&\!\!\al_2 m - \be_1 (m+1) + 1 + \frac{\al_2(m+1)}{\nu-1}
\\
&=\al_1 m - \be_1 (m+1) + 1 + \frac{\al_1(m+1)}{\nu-1}
  -\(m+\frac{m+1}{\nu-1}\)(\al_1-\al_2)
\\
&=\(p+1+\frac{p-q+2}{\nu-1}\)(m+1)
  -\(m+\frac{m+1}{\nu-1}\)\(q-\frac{h+1}{r}\)
\\
&=\(\frac{(\nu-1)(p+1)}{\nu}+\frac{p-q+2}{\nu}
  -\(\frac{(\nu-1)m}{\nu(m+1)}+\frac{1}{\nu}\)\(q-\frac{h+1}{r}\)\)
   \frac{\nu(m+1)}{\nu-1}
\\
&=\(p+1-\frac{q-1}{\nu}
  -\(\frac{m}{m+1}+\frac{1}{\nu(m+1)}\)\(q-\frac{h+1}{r}\)\)
   \frac{\nu(m+1)}{\nu-1}
\\
&\le 0. 
\end{align*}
It follows that \eqref{XiTh-thm1} and \eqref{ass1} are valid by Example \ref{Ex0} (ii) with $\al=\al_2$ and $\be=\be_1$. 
\end{proof}

%%%%%%%%%%%%%%%%%%%%%%%%%%%%%%%%%%%
\subsection{Volterra integral equation}
%%%%%%%%%%%%%%%%%%%%%%%%%%%%%%%%%%%

%%%%%%%%%%%%%%%%%%%%
\begin{lemma}\label{lemm-Vol} 
The solution to the Volterra integral equation 
\begin{equation}\label{Vol}
  w(t,\xi) = p(t,\xi) + \int^t_0 q(t,\tau,\xi) w(\tau,\xi)\,d\tau
\end{equation}
is formally represented as followis: 
\begin{align*}
  w(t,\xi)
 &=p(t,\xi) + \int^t_0 q(t,\tau_1,\xi)p(\tau_1,\xi)\,d\tau_1
\\
 &\;\; +\sum_{l=2}^\infty 
  \int^t_0 q(t,\tau_1,\xi)\int^{\tau_1}_0 q(\tau_1,\tau_2,\xi)\cdots
  \int^{\tau_{l-1}}_0 q(\tau_{l-1},\tau_l,\xi)p(\tau_l,\xi)\,d\tau_l
  \cdots d\tau_1. 
\end{align*}
\end{lemma}
%%%%%%%%%%%%%%%%%%%%
\begin{proof}
The proof can be verified by direct calculations. 
\end{proof}

%%%%%%%%%%%%%%%%%%%%%%%%%%%

\end{document}